\documentclass[11pt]{amsart}
\usepackage{amssymb,mathrsfs}
\newtheorem{theorem}{Theorem}[section]
\newtheorem{lemma}{Lemma}[section]
\newtheorem{proposition}{Proposition}[section]
\newtheorem{remark}{Remark}[section]
\newtheorem{definition}{Definition}[section]
\newtheorem{problem}{Problem}[section]
\newtheorem{conjecture}{Conjecture}[section]

\newcommand\charf {\mbox{{\text 1}\kern-.24em {\text l}}}
\newcommand\fg{{\mathfrak g}}

\newcommand\fp{{\mathfrak p}}
\newcommand\g{\gamma}
\newcommand\G{\Gamma}

\newcommand\lrt{\longrightarrow}

\newcommand\ba{\backslash}

\newcommand\CP{{\mathscr P}_n}
\newcommand\CM{{\mathcal M}}
\newcommand\BC{\mathbb C}
\newcommand\BZ{\mathbb Z}
\newcommand\BR{\mathbb R}
\newcommand\BQ{\mathbb Q}
\newcommand\BH{\mathbb H}
\newcommand\Rmn{{\mathbb R}^{(m,n)}}

\newcommand\Gnm{GL_{n,m}}

\newcommand\la{\lambda}

\newcommand\Yd{{{\partial}\over {\partial Y}}}
\newcommand\Vd{{{\partial}\over {\partial V}}}

\newcommand\BD{{\mathbb D}}
\newcommand\SP{{\mathfrak P}_n}
\newcommand\Om{{\Omega}}

\begin{document}

\title{Introduction to Automorphic Forms for $GL(n,\BZ)\ltimes \BZ^{(m,n)}$  }

\author{Jae-Hyun Yang}

\address{Yang Institute for Advanced Study
\newline\indent
Hyundai 41 Tower, No. 1905
\newline\indent
293 Mokdongdong-ro, Yangcheon-gu
\newline\indent
Seoul 07997, Korea
\vskip 2mm
and
\vskip 2mm
Department of Mathematics
\newline\indent
Inha University
\newline\indent
Incheon 22212, Korea}

\email{jhyang@inha.ac.kr\ \ or\ \ jhyang8357@gmail.com}

\thanks{2020 {\it Mathematics\ Subject\ Classification.} Primary 11Fxx, 11F55.
\endgraf
{\it Keywords\ and\ phrases.} Automorphic forms, Invariant differential operators,
Maass forms.}

\begin{abstract} In this paper, we introduce the notion of automorphic forms
for $GL(n,\BZ)\ltimes \BZ^{(m,n)}$ and discuss invariant differential operators on
the Minkowski-Euclid space. The group $GL(n,\BR)\ltimes \BR^{(m,n)}$ is
the semidirect product of $GL(n,\BR)$ and the additive group
$\BR^{(m,n)}$ and is {\it not} a reductive group. The Minkowski-Euclid space is
the quotient space of $GL(n,\BR)\ltimes \BR^{(m,n)}$ by $O(n,\BR)$.
The Minkowski-Euclid space is
an important non-symmetric homogeneous space geometrically and number theoretically.
We present some open problems to be solved in the future.
\end{abstract}
\maketitle

%%%%%%%%%%%%%%%%%%%%%%%%%%%%%%%%%%%%%%%%%%%%%%%%%%%%%%%%%%%%%%%%%%%%%%%%%%%%%%%%%%%%%%%%%%%%%%%%%%%%%%%%%%%%%%%%%%%%%%%%%%%%%%%%%%%%%%%%%%%%%%%%%%%%%%%%
%%%%%%%%%%%%%%%%%%%%%%%%%%%%%%%%%%%%%%%%%%%%%%%%%%%%%%%%%%%%%%%%%%%%%%%%%%%%%%%%%%%%%%%%%%%%%%%%%%%%%%%%%%%%%%%%%%%%%%%%%%%%%%%%%%%%%%%%%%%%%%%%%%%%%%%%
%%%%%%%%%%%%%%%%%%%%%%%%%%%%%%%%%%%%%%%%%%%%%%%%%%%%%%%%%%%%%%%%%%%%%%%%%%%%%%%%%%%%%%%%%%%%%%%%%%%%%%%%%%%%%%%%%%%%%%%%%%%%%%%%%%%%%%%%%%%%%%%%%%%%%%%%
%%%%%%%%%%%%%%%%%%%%%%%%%%%%%%%%%%%%%%%%%%%%%%%%%%%%%%%%%%%%%%%%%%%%%%%%%%%%%%%%%%%%%%%%%%%%%%%%%%%%%%%%%%%%%%%%%%%%%%%%%%%%%%%%%%%%%%%%%%%%%%%%%%%%%%%%
%%%%%%%%
%%%%%%%%
%%%%%%%%                   Section 1  Introduction
%%%%%%%%
%%%%%%%%
%%%%%%%%%%%%%%%%%%%%%%%%%%%%%%%%%%%%%%%%%%%%%%%%%%%%%%%%%%%%%%%%%%%%%%%%%%%%%%%%%%%%%%%%%%%%%%%%%%%%%%%%%%%%%%%%%%%%%%%%%%%%%%%%%%%%%%%%%%%%%%%%%%%%%%%%
%%%%%%%%%%%%%%%%%%%%%%%%%%%%%%%%%%%%%%%%%%%%%%%%%%%%%%%%%%%%%%%%%%%%%%%%%%%%%%%%%%%%%%%%%%%%%%%%%%%%%%%%%%%%%%%%%%%%%%%%%%%%%%%%%%%%%%%%%%%%%%%%%%%%%%%%
%%%%%%%%%%%%%%%%%%%%%%%%%%%%%%%%%%%%%%%%%%%%%%%%%%%%%%%%%%%%%%%%%%%%%%%%%%%%%%%%%%%%%%%%%%%%%%%%%%%%%%%%%%%%%%%%%%%%%%%%%%%%%%%%%%%%%%%%%%%%%%%%%%%%%%%%
%%%%%%%%%%%%%%%%%%%%%%%%%%%%%%%%%%%%%%%%%%%%%%%%%%%%%%%%%%%%%%%%%%%%%%%%%%%%%%%%%%%%%%%%%%%%%%%%%%%%%%%%%%%%%%%%%%%%%%%%%%%%%%%%%%%%%%%%%%%%%%%%%%%%%%%%

\begin{section}{{\bf Introduction}}
\setcounter{equation}{0}
\vskip 3mm
The theory of automorphic forms for $GL(n,\BZ)$ and $SL(n,\BZ)$ has been studied extensively
for more than seven decades. In 1956 Atle Selberg\,\cite{S1} initiated the study of
automorphic forms for
$GL(n,\BZ)$ and developed the theory of harmonic analysis on the symmetric space
$GL(n,\BR)/O(n,\BR)$. We recommend two books \cite{Go} and \cite{T} for the details
of the theory of automorphic forms for $GL(n,\BZ)$ and $SL(n,\BZ)$.
\vskip 2mm
For a positive integer $n\geq 1$, we let
\begin{equation*}
  {\mathscr P}_n:=\{\,Y\in GL(n,\BR)\,|\ Y=\,{}^tY > 0\,\}
\end{equation*}
be the open convex cone in the Euclidean space $\BR^N$ with $N=\frac{n(n+1)}{2}$.
For any two positive integers $m,\,n\in\BZ^+$, we let
\begin{equation*}
  {\mathscr P}_{n,m}:={\mathscr P}_n\times \BR^{(m,n)}
\end{equation*}
be the so-called {\sf Minkowski-Euclid space} of type $(n,m)$\ (cf.\,\cite{Y3}).
\vskip 2mm
For any two positive integers $m,\,n\in\BZ^+$, we introduce the new group
$$ GL_{n,m}(\BR):=GL(n,\BR)\ltimes \BR^{(m,n)}$$
with multiplication law
\begin{equation}\label{Formula (1.1)}
  (A,a)\circ (B,b):=(AB,a\,{}^tB^{-1}+b)
\end{equation}
for all $A,B\in GL(n,\BR)$ and $a,b\in\BR^{(m,n)}$.
We note that $GL_{n,m}(\BR)$ is the semidirect product of $GL(n,\BR)$ and the
additive group $\BR^{(m,n)}.$ Then $GL_{n,m}(\BR)$ acts on ${\mathscr P}_{n,m}$
naturally and transitively by
\begin{equation}\label{Formula (1.2)}
  (A,a)\cdot (Y,V):=(AY\,{}^t\!A,(V+a)\,{}^tA)
\end{equation}
for all $(A,a)\in GL_{n,m}(\BR)$ and $(Y,V)\in {\mathscr P}_{n,m}.$
Since $O(n,\BR)$ is the stabilizer of the action (1.2) at $(I_n,0)$,
the non-symmetric homogeneous space $GL_{n,m}(\BR)/O(n,\BR)$ is diffeomorphic to
the Minkowski-Euclid space ${\mathscr P}_{n,m}$. We denote by $\BD({\mathscr P}_{n,m})$
the algebra of all differential operators on ${\mathscr P}_{n,m}$ invariant under the
action (1.2) of $GL_{n,m}(\BR)$. Using the subalgebra of $\BD({\mathscr P}_{n,m})$
containing the Laplace operator $\Delta_{n,m}$ of ${\mathscr P}_{n,m}$, we define
the concept of automorphic forms for $GL_{n,m}(\BZ)$. Here
$$GL_{n,m}(\BZ):=GL(n,\BZ)\ltimes \BZ^{(m,n)}$$
denotes the discrete subgroup of $GL_{n,m}(\BR)$.
\vskip 2mm
Let
$$ SL_{n,m}(\BR):=SL(n,\BR)\ltimes \BR^{(m,n)}$$
be the subgroup of $GL_{n,m}(\BR)$ and also let
$${\mathfrak P}_{n,m}:=\{\,(Y,V)\in {\mathscr P}_{n,m}\,|\ \det Y=1\,\}$$
be the subset of ${\mathscr P}_{n,m}$. Then $SL_{n,m}(\BR)$ acts on
${\mathfrak P}_{n,m}$ transitively via the formula (1.2). Since $SO(n,\BR)$
is the stabilizer of the action (1.2) at $(I_n,0)$,
the non-symmetric homogeneous space $SL_{n,m}(\BR)/SO(n,\BR)$ is diffeomorphic to
the manifold ${\mathfrak P}_{n,m}$. We denote by $\BD({\mathfrak P}_{n,m})$
the algebra of all differential operators on ${\mathfrak P}_{n,m}$ invariant under the
action (1.2) of $SL_{n,m}(\BR)$.

\vskip 2mm
The group $GL_{n,m}(\BR)$ is a natural group and its action (1.2) on ${\mathscr P}_{n,m}$
is a natural action. We will explain how the action (1.2) arises. Let
$$
Sp(2n,\BR)=\{ \,M\in GL(2n,\BR)\,|\ {}^tMJ_nM=J_n\,\}
$$
be the symplectic group of degree $n$, where
$$
J_n=
\begin{pmatrix}
  \,0 & I_n \\
  -I_n & 0
\end{pmatrix}
$$
is the symplectic matrix of degree $n$. We observe that $Sp(2n,\BR)$ is a simple
Lie group. Let
$$
\BH_n=\{ \Om\in \BC^{(n,n)}\,|\ \Om={}^t\Om,\ {\rm Im}\,\Om >0\,\}
$$
be the Siegel upper half plane or simply the Siegel space of degree $n$. Then
$Sp(2n,\BR)$ acts on $\BH_n$ transitively by
\begin{equation}\label{Foemula (1.3)}
  M\langle \Om \rangle =(A\Om +B)(C\Om +D)^{-1},
\end{equation}
where $M=\begin{pmatrix}
           A & B \\
           C & D
         \end{pmatrix}\in Sp(2n,\BR)$ and $\Om\in \BH_n.$
Since the unitary group $U(n)$ is the stabilizer of the action (1.3) at $i\,I_n$,
the symmetric space $Sp(2n,\BR)/U(n)$ is biholomorphic to the Siegel space $\BH_n$.
It is known that the Siegel space $\BH_n$ is an Einstein-K{\"a}hler Hermitian symmetric
space. Indeed $\BH_n$ is an important complex manifold geometrically and number
theoretically. Let $\G_n^\flat=Sp(2n,\BZ)$ be the Siegel modular group of degree $n$.
The Siegel modular variety $\mathscr{A}_n:=\G_n^\flat \ba \BH_n$ is a quasi-projective
variety and may be regarded as the moduli of principally polarized abelian varieties of
dimension $n$. The Satake compactification $\mathscr{A}_n^S$ is a normal projective
variety but is highly singular. The toroidal compactification of $\mathscr{A}_n$
obtained by the Mumford school is a smooth projective variety.
The action (1.3) induces the natural and transitive action of $GL(n,\BR)$ on
the open convex cone $\mathscr{P}_n$ given by
\begin{equation}\label{Foemula (1.4}
  g\cdot Y:= gY\,{}^tg,\qquad {\rm where}\ g\in GL(n,\BR)\ {\rm and}\ Y\in \mathscr{P}_n.
\end{equation}
\vskip 2mm
The quotient space $\mathfrak{T}_n:=GL(n,\BZ)\ba \mathscr{P}_n$ may be regarded as
the moduli of principally polarized real tori of dimension $n$. A compactification of
$\mathfrak{T}_n$ was given by Grenier (cf.\,\cite{Gr2, Gr3}).
\vskip 2mm
The relations between the cone $\mathscr{P}_n$ and the Siegel space $\BH_n$ may be described
as follows\,:
\vskip 2mm
\ \ \ The open cone $\mathscr{P}_n\qquad\qquad\hskip 30mm\qquad {\rm The\ Siegel\ space}\ \BH_n$
\vskip 0.1mm\noindent
---------------------------------------------------------------------------------------------------------------------
\begin{eqnarray*}
  GL(n,\BR)\qquad \qquad &{}& \qquad Sp(2n,\BR) \\
  \ O(n) \qquad \qquad &{}& \qquad U(n) \\
  {\rm open\ convex\ cone}  \qquad \qquad &{}& \qquad
  {\rm Einstein\!\!-\!\!Kaehler\ Hermitian\ symmetric}\\
  {\rm Minkowski\ fundamental\ domain}  \qquad \qquad &{}& \qquad
  {\rm Siegel\ fundamental\  domain}\\
  {\rm pprts} \qquad \qquad &{}& \qquad {\rm ppavs}\\
  {\rm Grenier's\ compactification} \qquad \qquad &{}& \qquad {\rm Satake\ compactification}\\
  \qquad \qquad &{}& \qquad {\rm toroidal\ compactification}\\
  {\rm automorphic\ forms\ for}\ GL(n,\BZ) \qquad \qquad &{}& \qquad
  {\rm Siegel\ modular\ forms,\ Siegel\!\!-\!\!Maass\ forms}
\end{eqnarray*}
\vskip 3mm\noindent
Here ``pprt (resp.\,ppav)" means the principally polarized real torus (resp.\, principally
polarized abelian variety). We may say that $\mathscr{P}_n$ is a real version of $\BH_n$. 
The algebra $\BD (\mathscr{P}_n)$ of invariant differential operators on $\mathscr{P}_n$ is 
commutative and also the algebra $\BD (\BH_n)$ of invariant differential operators on 
$\BH_n$ is commutative.

\vskip 3mm
We recall the Jacobi group
$$
G^J:=Sp(2n,\BR)\ltimes H_\BR^{(n,m)} ,
$$
where
$$
H_\BR^{(n,m)}=\{ (\la,\mu;\kappa)\,|\ \la,\mu\in\BR^{(m,n)},\ \kappa\in \BR^{(m,m)},\
\kappa+\mu\,{}^t\la\ \,\rm{symmetric}\ \}
$$
is the Heisenberg group endowed with the following multiplication
\begin{equation}\label{Formula (1.5)}
(\la,\mu;\kappa)\circ (\la',\mu';\kappa')=(\la+\la',\mu+\mu';\kappa+\kappa'
+\la\,{}^t\mu'-\mu\,{}^t\la')
\end{equation}
with $(\la,\mu;\kappa),\,(\la',\mu';\kappa')\in H_\BR^{(n,m)}.$
The Jacobi group $G^J$ of degree $n$ and index $m$ is the semidirect product of
$Sp(2n,\BR)$ and $H_{\BR}^{(n,m)}$
endowed with the following multiplication law
\begin{equation}\label{Formula (1.6)}
\big(M,(\lambda,\mu;\kappa)\big)\cdot\big(M',(\lambda',\mu';\kappa'\,)\big)
=\, \big(MM',(\tilde{\lambda}+\lambda',\tilde{\mu}+ \mu';
\kappa+\kappa'+\tilde{\lambda}\,^t\!\mu'
-\tilde{\mu}\,^t\!\lambda'\,)\big)
\end{equation}
with $M,M'\in Sp(2n,\BR),
(\lambda,\mu;\kappa),\,(\lambda',\mu';\kappa') \in H_{\BR}^{(n,m)}$ and
$(\tilde{\lambda},\tilde{\mu})=(\lambda,\mu)M'$. Then $G^J$ acts
on $\BH_n\times \BC^{(m,n)}$ transitively by
\begin{equation}
\big(M,(\lambda,\mu;\kappa)\big)\cdot
(\Om,Z)=\Big(M\langle\Om\rangle,(Z+\lambda \Om+\mu)
(C\Omega+D)^{-1}\Big),
\end{equation}
where
$M=\begin{pmatrix} A&B\\
C&D\end{pmatrix} \in Sp(2n,\BR),\ (\lambda,\mu; \kappa)\in
H_{\BR}^{(n,m)}$ and $(\Om,Z)\in \BH_n\times \BC^{(m,n)}.$ We note
that the Jacobi group $G^J$ is {\it not} a reductive Lie group. Since
the unitary group $U(n)$ of degree $n$ is the stabilizer of the action (1.7),
the homogeneous space $G^J/U(n)$ is biholomorphic to the complex manifold
${\mathbb H}_n\times \BC^{(m,n)}$ and is not a
symmetric space. From now on, for brevity we write
$\BH_{n,m}=\BH_n\times \BC^{(m,n)}.$ The homogeneous space $\BH_{n,m}$ is called the
{\it Siegel-Jacobi space} of degree $n$ and index $m$. Since $\BH_{n,m}$ is a
K{\"a}hler manifold, it is a symplectic manifold.
We set
$$
H_\BZ^{(n,m)}=\{\,(\lambda,\mu;\kappa)\in H_\BR^{(n,m)}\,|\ \lambda,\mu,\kappa\
{\rm are\ integral}\,\}
$$
and
$$\Gamma^J:=Sp(2n,\BZ)\ltimes H_\BZ^{(n,m)}.$$
Then $\mathscr{A}_{n,m}:=\G^J\ba \BH_{n,m}$ is the universal family of principally polarized
abelian varieties over $\mathscr{A}_n$.
Roughly $\mathscr{A}_{n,m}$ is the fibration over $\mathscr{A}_n$ with fibres as
principally polarized abelian varieties. We refer to \cite{Y5,Y8} for more details on
the geometry of the Siegel-Jacobi space $\BH_{n,m}$.
\vskip 2mm
The action (1.7) of $G^J$ on $\BH_{n,m}$ induces the natural and transitive action of
$GL_{n,m}(\BR)$ on the Minkowski-Euclid space $\mathscr{P}_{n,m}.$ We refer to
\cite{Y3, Y4, Y9} for more details on the geometry of the Minkowski-Euclid space.
$\mathfrak{T}_{n,m}:=GL_{n,m}(\BZ)\ba \mathscr{P}_{n,m}$ is the universal real torus
over the moduli space $GL(n,\BZ)\ba \mathscr{P}_n.$ Roughly $\mathfrak{T}_{n,m}$ is
a fibration over $GL(n,\BZ)\ba \mathscr{P}_n$ with fibres as principally polarized
real tori.
\vskip 2mm
The relations between the Minkowski-Euclid space $\mathscr{P}_{n,m}$ and
the Siegel-Jacobi space $\BH_{n,m}$ as follows:
\vskip 3mm\noindent
\ {\rm The\ Minkowski-Euclid\ space}\ $\mathscr{P}_{n,m}\qquad\qquad
\hskip 15mm  {\rm The\ Siegel\!\!-\!\!Jacobi\ space}\ \BH_{n,m}$
\vskip 0.1mm\noindent
----------------------------------------------------------------------------------------------------------------------
\begin{eqnarray*}
  GL(n,\BR)\ltimes \BR^{(m,n)}\qquad \qquad &{}& \qquad Sp(2n,\BR)\ltimes H_\BR^{(n,m)} \\
  \ O(n) \qquad \qquad &{}& \qquad U(n) \\
  {\rm non\!\!-\!\!symmetric\ space}  \qquad \qquad &{}& \qquad
  {\rm Kaehler\ symplectic\ space}\\
  {\rm the\ universal\ real\ torus} \qquad \qquad &{}& \qquad
  {\rm the\ universal\ abelian\ variety} \\
  {\rm pprts} \qquad \qquad &{}& \qquad {\rm ppavs}\\
   \qquad \qquad &{}& \qquad {\rm toroidal\ compactifications}\\
  {\rm automorphic\ forms\ for}\ GL_{n,m}(\BZ) \qquad \qquad &{}& \qquad
  {\rm Jacobi\ forms,\ Maass\!\!-\!\!Jacobi\ forms}\\
  \BD (\mathscr{P}_{n,m})\ {\rm is\ not\ commutative}
  \qquad \qquad &{}& \qquad
  \BD(\BH_{n,m}) \ {\rm is\ not\ commutative}
\end{eqnarray*}
\vskip 3mm\noindent
Here $\BD (\mathscr{P}_{n,m})$ denotes the algebra of all differential operators
on $\mathscr{P}_{n,m}$ invariant under the action (1.2) of $GL_{n,m}(\BR)$ on
$\mathscr{P}_{n,m}$ and $\BD(\BH_{n,m})$ denotes the algebra of all differential operators
on $\BH_{n,m}$ invariant under the action (1.7) of $G^J$ on $\BH_{n,m}$.
\vskip 3mm
In this sense, $GL_{n,m}(\BR)$ and $\mathscr{P}_{n,m}$ are important geometrically and
number theoretically. We note that the theory of automorphic forms on $\mathscr{P}_{n,m}$
for $GL_{n,m}(\BZ)$ generalizes the theory of automorphic forms on $\mathscr{P}_n$ for
$GL(n,\BZ).$ We propose to study the harmonic analysis on the Hilbert spaces
$L^2 \left( GL_{n,m}(\BZ)\ba \mathscr{P}_{n,m}\right)$ and
$L^2\left( GL_{n,m}(\BZ)\ba GL_{n,m}(\BR)\right)$.
\vskip 2mm
So far we explained why we should study automorphic forms on $\mathscr{P}_{n,m}$ for
$GL_{n,m}(\BR)$ and develop the theory of harmonic analysis on the Hilbert spaces
$L^2 \left( GL_{n,m}(\BZ)\ba \mathscr{P}_{n,m}\right)$ and
$L^2\left( GL_{n,m}(\BZ)\ba GL_{n,m}(\BR)\right)$.

\vskip 2mm
The aim of this article is to introduce new automorphic forms for $GL_{n,m}(\BZ)$
and $SL_{n,m}(\BZ)$ generalizing those for $GL(n,\BZ)$ and $SL(n,\BZ)$ and present some
open problems to be studied and solved in the future.
\vskip 2mm
This article is organized as follows. In Section 2, we give a brief review of automorphic
forms for $GL(n,\BZ)$. The theory of automorphic forms for $GL(n,\BZ)$ has been studied by Selberg, Maass,
Terras, Bump, Goldfeld and other mathematicians\,(cf.\,\cite{S1,M2,T,Go}).
The adelic version of this theory has been
studied by Langlands school. We recommend the book \cite{T} of Audrey Terras which deals extensively
with the theory of automorhic forms for $GL(n,\BZ)$ real analytically in some details.
In Section 3, we give a brief review of automorphic forms for $SL(n,\BZ)$.
The theory of automorphic forms for $SL(n,\BZ)$ has been studied by Goldfeld, Jacquet and other number
theorists. The book \cite{Go} of Goldfeld deals with many topics related to automorphic forms for
$SL(n,\BZ).$ In Section 4, we introduce the concept of new automorphic forms for
$GL(n,\BZ)\ltimes\BZ^{(m,n)}$ using the commutative subalgebra of $\BD({\mathscr P}_{n,m})$ containing
the Laplace operator of the Minkowski-Euclid space ${\mathscr P}_{n,m}$. We present some open
problems related to $GL_{n,m}(\BR)$-invariant differential operators on ${\mathscr P}_{n,m}$.
We propose to study the theory of new automorphic forms for $GL(n,\BZ)\ltimes\BZ^{(m,n)}$
in order to develop the theory of harmonic analysis on the $L^2$-space
$L^2(GL_{n,m}(\BZ)\backslash {\mathscr P}_{n,m})$. We also propose to
develop the theory of harmonic analysis on the $L^2$-space
$L^2(GL_{n,m}(\BZ)\backslash GL_{n,m}(\BR))$ representation theoretically.
In Section 5, we introduce the concept of new automorphic forms for
$SL(n,\BZ)\ltimes\BZ^{(m,n)}$ using the commutative subalgebra of $\BD({\mathfrak P}_{n,m})$ containing
the Laplace operator of the non-symmetric homogeneous space ${\mathfrak P}_{n,m}$. We present some open
problems related to $SL_{n,m}(\BR)$-invariant differential operators on ${\mathfrak P}_{n,m}$.
For brevity, we put
$$ SL_{n,m}(\BZ):=SL(n,\BZ)\ltimes \BZ^{(m,n)}.$$
We propose to study the theory of new automorphic forms for $SL_{n,m}(\BZ)$
in order to develop the theory of harmonic analysis on the $L^2$-space
$L^2(SL_{n,m}(\BZ)\backslash {\mathfrak P}_{n,m})$. We also propose to
develop the theory of harmonic analysis on the $L^2$-space
$L^2(SL_{n,m}(\BZ)\backslash SL_{n,m}(\BR))$ representation theoretically.

\vskip 3mm
We hope that this article will open a starting point to study new automorphic forms
for $GL_{n,m}(\BZ)$ and $SL_{n,m}(\BZ)$ covering their related $L$-functions, Whittaker
functions, Hecke operators, Eisenstein series, Poincar{\'e} series and so on.
In the near future we hope to understand the beautiful theory of harmonic analysis on
the four $L^2$-space $L^2(GL_{n,m}(\BZ)\backslash {\mathscr P}_{n,m})$,
$L^2(GL_{n,m}(\BZ)\backslash GL_{n,m}(\BR))$, $L^2(SL_{n,m}(\BZ)\backslash {\mathfrak P}_{n,m})$
and $L^2(SL_{n,m}(\BZ)\backslash SL_{n,m}(\BR))$

\vskip 0.31cm \noindent {\bf Notations:} \ \ We denote by
$\BQ,\,\BR$ and $\BC$ the field of rational numbers, the field of
real numbers and the field of complex numbers respectively. We
denote by $\BZ$ and $\BZ^+$ the ring of integers and the set of
all positive integers respectively. $\BR^{\times}$ (resp. $\BC^{\times}$)
denotes the group of nonzero real (resp. complex) numbers.
The symbol ``:='' means that
the expression on the right is the definition of that on the left.
For two positive integers $k$ and $l$, $F^{(k,l)}$ denotes the set
of all $k\times l$ matrices with entries in a commutative ring
$F$. For a square matrix $A\in F^{(k,k)}$ of degree $k$,
${\rm Tr}(A)$ denotes the trace of $A$. For any $M\in F^{(k,l)},\
^t\!M$ denotes the transpose of $M$. For a positive integer $n$, $I_n$
denotes the identity matrix of degree $n$.
For $A\in F^{(k,l)}$ and $B\in
F^{(k,k)}$, we set $B[A]=\,^tABA$ (Siegel's notation). For a complex matrix $A$,
${\overline A}$ denotes the complex {\it conjugate} of $A$.
${\rm diag}(a_1,\cdots,a_n)$ denotes the $n\times n$ diagonal matrix with diagonal entries
$a_1,\cdots,a_n$. For a smooth manifold, we denote by $C_c (X)$ (resp. $C_c^{\infty}(X)$
the algebra of all continuous (resp. infinitely differentiable) functions on $X$ with compact support.
$O(n):=O(n,\BR)=\{\,g\in GL(n,\BR)\,|\ g\,{}^tg={}^tgg=I_n\,\}$
is the real orthogonal matrix of degree $n$. $SO(n):=SO(n,\BR)=O(n)\cap SL(n,\BR).$
\vskip 2mm
We denote
\begin{align*}
GL_{n,m}(\BR)\! &=\! GL(n,\BR)\ltimes \BR^{(m,n)}, \quad
GL_{n,m}(\BZ)= GL(n,\BZ)\ltimes \BZ^{(m,n)}, 
\end{align*}

\begin{align*}
SL_{n,m}(\BR)\! &=\! SL(n,\BR)\ltimes \BR^{(m,n)},  \quad
SL_{n,m}(\BZ)=SL(n,\BZ)\ltimes \BZ^{(m,n)}, \\ 
\mathscr{P}_n\! &=\!\{ Y\in \BR^{(n,n)}\,|\ Y=\,{}^tY>0\ \}
\cong GL(n,\BR)/O(n,\BR),\\
\mathfrak{P}_n\! &=\!\{ Y\in \BR^{(n,n)}\,|\ Y=\,{}^tY>0,\ \det\,(Y)=1\ \}
\cong SL(n,\BR)/SO(n,\BR),\\
\mathfrak{H}_n\! &{ }\ ({\rm see\ Definition\ 3.1}),\\
\mathscr{P}_{n,m}\! &=\! \mathscr{P}_n \times \BR^{(m,n)}\cong GL_{n,m}(\BR)/O(n,\BR), \\
\mathfrak{P}_{n,m}\! &=\! \mathfrak{P}_n \times \BR^{(m,n)}\cong SL_{n,m}(\BR)/SO(n,\BR), \\
\mathfrak{H}_{n,m}\! &=\! \mathfrak{H}_n \times \BR^{(m,n)}\cong SL_{n,m}(\BR)/SO(n,\BR), \\
\G_n\!&=\! GL(n,\BZ),\quad \G^n=SL(n,\BZ),\\
\G_{n,m}\! &=\!\G_n\ltimes \BZ^{(m,n)}= GL_{n,m}(\BZ),\\
\G^{n,m}\! &=\!\G^n\ltimes \BZ^{(m,n)}= SL_{n,m}(\BZ),\\
\mathfrak{Y}_n\!&=\! \G_n\ba \mathscr{P}_n\cong GL(n,\BZ)\ba GL(n,\BR)/O(n,\BR),   \\
\mathfrak{X}_n\!&=\! \G^n\ba \mathfrak{P}_n\cong SL(n,\BZ)\ba SL(n,\BR)/SO(n,\BR),  \\
\mathfrak{Y}_{n,m}\!&=\! \G_{n,m}\ba \mathscr{P}_{n,m}\cong
GL_{n,m}(\BZ)\ba GL_{n,m}(\BR)/O(n,\BR), \\
\mathfrak{X}_{n,m}\!&=\! \G^{n,m}\ba \mathfrak{P}_{n,m}\cong
SL_{n,m}(\BZ)  \ba SL_{n,m}(\BR)/SO(n,\BR).
\end{align*}
%\vskip 1mm\noindent
Here $``\cong"$ denotes the diffeomorphism.
$\BD (\mathscr{P}_{n,m})$ denotes the algebra of all $GL_{n,m}(\BR)$-invariant
differential operators on $\mathscr{P}_{n,m}$.
$\BD (\mathfrak{P}_{n,m})$ denotes the algebra of all $SL_{n,m}(\BR)$-invariant
differential operators on $\mathfrak{P}_{n,m}$.
$\BD (\mathfrak{H}_{n,m})$ denotes the algebra of all $SL_{n,m}(\BR)$-invariant
differential operators on $\mathfrak{H}_{n,m}$.
$\mathscr{Z}_{n,m}$ denotes the center of $\BD (\mathscr{P}_{n,m})$.
$\mathfrak{Z}_{n,m}$ denotes the center of $\BD (\mathfrak{P}_{n,m})$ and
$\widetilde{\mathfrak{Z}}_{n,m}$ denotes the center of $\BD (\mathfrak{H}_{n,m})$.

\end{section}

\newcommand\POB{ {{\partial}\over {\partial{\overline \Omega}}} }
\newcommand\PZB{ {{\partial}\over {\partial{\overline Z}}} }
\newcommand\PX{ {{\partial}\over{\partial X}} }
\newcommand\PY{ {{\partial}\over {\partial Y}} }
\newcommand\PU{ {{\partial}\over{\partial U}} }
\newcommand\PV{ {{\partial}\over{\partial V}} }
\newcommand\PO{ {{\partial}\over{\partial \Omega}} }
\newcommand\PZ{ {{\partial}\over{\partial Z}} }
\newcommand\PW{ {{\partial}\over{\partial W}} }
\newcommand\PWB{ {{\partial}\over {\partial{\overline W}}} }
\newcommand\OVW{\overline W}
\newcommand\Rg{{\mathfrak R}_n}

%%%%%%%%%%%%%%%%%%%%%%%%%%%%%%%%%%%%%%%%%%%%%%%%%%%%%%%%%%%%%%%%%%%%%%%%%%%%%%%%%%%%%%%%%%%%%%%%%%%%%%%%%%%%%%%%%%%%%%%%%%%%%%%%%%%%%%%%%%%%%%%%%%%%%%%%
%%%%%%%%%%%%%%%%%%%%%%%%%%%%%%%%%%%%%%%%%%%%%%%%%%%%%%%%%%%%%%%%%%%%%%%%%%%%%%%%%%%%%%%%%%%%%%%%%%%%%%%%%%%%%%%%%%%%%%%%%%%%%%%%%%%%%%%%%%%%%%%%%%%%%%%%
%%%%%%%%%%%%%%%%%%%%%%%%%%%%%%%%%%%%%%%%%%%%%%%%%%%%%%%%%%%%%%%%%%%%%%%%%%%%%%%%%%%%%%%%%%%%%%%%%%%%%%%%%%%%%%%%%%%%%%%%%%%%%%%%%%%%%%%%%%%%%%%%%%%%%%%%
%%%%%%%%%%%%%%%%%%%%%%%%%%%%%%%%%%%%%%%%%%%%%%%%%%%%%%%%%%%%%%%%%%%%%%%%%%%%%%%%%%%%%%%%%%%%%%%%%%%%%%%%%%%%%%%%%%%%%%%%%%%%%%%%%%%%%%%%%%%%%%%%%%%%%%%%
%%%%%%%%
%%%%%%%%
%%%%%%%%                   2. Automorphic Forms for $GL(n,\BZ)
%%%%%%%%
%%%%%%%%
%%%%%%%%%%%%%%%%%%%%%%%%%%%%%%%%%%%%%%%%%%%%%%%%%%%%%%%%%%%%%%%%%%%%%%%%%%%%%%%%%%%%%%%%%%%%%%%%%%%%%%%%%%%%%%%%%%%%%%%%%%%%%%%%%%%%%%%%%%%%%%%%%%%%%%%%
%%%%%%%%%%%%%%%%%%%%%%%%%%%%%%%%%%%%%%%%%%%%%%%%%%%%%%%%%%%%%%%%%%%%%%%%%%%%%%%%%%%%%%%%%%%%%%%%%%%%%%%%%%%%%%%%%%%%%%%%%%%%%%%%%%%%%%%%%%%%%%%%%%%%%%%%
%%%%%%%%%%%%%%%%%%%%%%%%%%%%%%%%%%%%%%%%%%%%%%%%%%%%%%%%%%%%%%%%%%%%%%%%%%%%%%%%%%%%%%%%%%%%%%%%%%%%%%%%%%%%%%%%%%%%%%%%%%%%%%%%%%%%%%%%%%%%%%%%%%%%%%%%
%%%%%%%%%%%%%%%%%%%%%%%%%%%%%%%%%%%%%%%%%%%%%%%%%%%%%%%%%%%%%%%%%%%%%%%%%%%%%%%%%%%%%%%%%%%%%%%%%%%%%%%%%%%%%%%%%%%%%%%%%%%%%%%%%%%%%%%%%%%%%%%%%%%%%%%%
\vskip 10mm
\begin{section}{{\bf Automorphic Forms for $GL(n,\BZ)$} }
\setcounter{equation}{0}
\vskip 3mm
In this section, we give a brief review of automorphic forms for $GL(n,\BZ).$
We refer to \cite{S1,M2,T} for more details.
\vskip 2mm
Let
\begin{equation*}
  {\mathscr P}_n:=\{\,Y\in \BR^{(n,n)}\,|\ Y=\,{}^tY>0 \,\}
\end{equation*}
be the open convex cone in $\BR^N$ with $N=\frac{n(n+1)}{2}$. Then the general linear group
$GL(n,\BR)$ acts on ${\mathscr P}_n$ transitively by

\begin{equation}\label{Formula (2.1)}
  g\circ Y=gY\,{}^tg\qquad {\rm for\ all}\ g\in GL(n,\BR)\ {\rm and}\ Y\in {\mathscr P}_n.
\end{equation}
Since $O(n,\BR)$ is the stabilizer of the action (2.1) at $I_n$, the symmetric space
$GL(n,\BR)/O(n,\BR)$ is diffeomorphic to ${\mathscr P}_n$ via
$$  g\cdot O(n,\BR)\longmapsto g\circ I_n=g\,{}^tg,\quad g\in GL(n,\BR).$$

We briefly review some geometrical properties on ${\mathscr P}_n.$
\vskip 0.3cm
For a coordinate $Y=(y_{ij})\in \CP,$ we put
\begin{equation*}
dY=(dy_{ij})\qquad\text{and}\qquad \PY\,=\,\left(\, {
{1+\delta_{ij}}\over 2}\, { {\partial}\over {\partial y_{ij} } }
\,\right).
\end{equation*}

\vskip 0.2cm For a fixed element $A\in GL(n,\BR)$, we put
$$Y_*=A\circ Y=AY\,^t\!A,\quad Y\in \CP.$$
Then
\begin{equation}
dY_*=A\,dY\,^t\!A \quad \textrm{and}\quad {{\partial}\over {\partial
Y_*}}=\,^t\!A^{-1} \Yd\, A^{-1}.
\end{equation}

\vskip 5mm
For a positive integer $c$, we can see easily that
\begin{equation*}  ds^2_{n;c}=c\cdot {\rm Tr}( (Y^{-1}dY)^2)  \end{equation*}
is a $GL(n,\BR)$-invariant Riemannian metric on $\CP$ and its
Laplacian is given by
\begin{equation*}
\Delta_{n;c}=\frac{1}{c} \cdot  {\rm Tr}\left( \left( Y\PY\right)^2\right),
\end{equation*}
\noindent where $ {\rm Tr}(M)$ denotes the trace of a square
matrix $M$.
\newcommand\da{ {\partial}\over {\partial y_1} }
\newcommand\db{ {\partial}\over {\partial y_2} }
\newcommand\dc{ {\partial}\over {\partial y_3} }
\vskip 2mm
For instance, we consider the case $n=2$ and $A>0$. If we write
for $Y\in {\mathscr P}_2$
$$Y=\begin{pmatrix} y_1 & y_3 \\ y_3 & y_2 \end{pmatrix}\quad
\text{and}\quad {{\partial}\over {\partial Y}}=\begin{pmatrix} \da & {\frac 12} {\dc} \\
{\frac 12} {\dc} & \db \end{pmatrix},$$ then

\begin{eqnarray*}
 ds^2_{2;A}&=& A\cdot \textrm{Tr}\big(Y^{-1}dY\,Y^{-1}dY\big)\cr
 &=&{\frac A{\big(y_1y_2-y^2_3\big)^2}}\,\,  \Big\{ y_2^2\,
 dy_1^2\,+\,y_1^2\,dy_2^2\,+\,2\,\big(y_1y_2+y_3^2\big)\,dy_3^2 \cr & & \ \ \ \
 \ \ \
 +\,2\,y_3^2 \,\,dy_1dy_2- 4\,y_2\,y_3\,\,dy_1dy_3 -4\,y_1y_3\,\,dy_2 dy_3
\Big\}
\end{eqnarray*}

and its Laplace operator $\Delta_{2;A}$ on ${\mathscr P}_2$ is

\begin{eqnarray*}
 \Delta_{2;A}&=& {\frac 1A}\cdot \textrm{Tr}  \left( \left( Y\PY\right)^2\right)\cr
&=&{\frac 1A}\, \Biggl\{ y_1^2\,{{\partial^2}\over{\partial
y_1^2}}+y_2^2\,{{\partial^2}\over{\partial y_2^2}}+ {\frac
12}(y_1y_2+y_3^2){{\partial^2}\over{\partial y_3^2}}\cr & &\ \ \
+2\left( y_3^2 {{\partial^2}\over {\partial y_1\partial y_2}}+
y_1y_3 {{\partial^2}\over {\partial y_1\partial y_3}}+y_2y_3
{{\partial^2}\over {\partial y_2\partial y_3}} \right)\cr & &\ \ \
+ {3 \over 2}\left( y_1{ {\partial}\over {\partial y_1} }+y_2{
{\partial}\over {\partial y_2} }+y_3{ {\partial}\over {\partial
y_3} }\right)\Biggl\}.
\end{eqnarray*}

We also can see that
\begin{equation}
dv_n(Y)=(\det Y)^{-{ {n+1}\over2 } }\prod_{i\leq j}dy_{ij}
\end{equation}
is a $GL(n,\BR)$-invariant volume element on $\CP$ (cf. \cite[p.\,23, see (1.21)]{T}).

\begin{theorem}
A geodesic $\alpha (t)$ joining $I_n$ and $Y\in \CP$ has the form
\begin{equation*}
\alpha (t)=\exp (t A[k]),\qquad t\in [0,1],
\end{equation*}
where
\begin{equation*}
Y=(\exp A)[k]=\exp (A[k])=\exp (\,^tkAk)
\end{equation*}
is the spectral decomposition of $Y$, where $k\in O(n,\BR),\ A={\rm diag} (a_1,\cdots,a_n)$ with all $a_j\in \BR.$
The distance of $\alpha (t) \ (0\leq t\leq 1)$ between $I_n$ and $Y$ is
\begin{equation*}
 \left( \sum_{j=1}^{n} a_j^2 \right)^{\frac{1}{2}}.
\end{equation*}
\end{theorem}
\begin{proof} The proof can be found in \cite[pp.\,16-17]{T}.\end{proof}

\vskip 2mm
We consider the so-called {\sf Maass-Selberg (differential) operators}
\begin{equation}
\delta_j:= Tr\left( \left( Y\Yd \right)^j\right),\quad j=1,2,\cdots,n.
\end{equation}
By Formula (2.2), we get
\begin{equation*}
\left( Y_* {{\partial}\over {\partial Y_*}}\right)^j=\,A\,\left(
Y\Yd\right)^j A^{-1},\quad j=1,2,\cdots,n.
\end{equation*}

\noindent for any $A\in GL(n,\BR)$. So each $\delta_j\ (1\leq j \leq n)$ is invariant
under the action (2.1) of $GL(n,\BR)$.

\vskip 0.2cm Selberg \cite{S1} proved the following.

\begin{theorem}
The algebra ${\mathbb D}(\CP)$ of all $GL(n,\BR)$-invariant differential operators on
$\CP$ is generated by $\delta_1,\delta_2,\cdots,\delta_n.$ Furthermore
$\delta_1,\delta_2,\cdots,\delta_n$ are algebraically independent and ${\mathbb D}(\CP)$
is isomorphic to the commutative ring $\BC[x_1,x_2,\cdots,x_n]$ with $n$ indeterminates $x_1,x_2,\cdots,x_n.$
\end{theorem}

\begin{proof} The proof can be found in \cite[pp.\,64-66]{M2}.\end{proof}

\vskip 3mm
For $s=(s_1,\cdots,s_n)\in \BC^n$, Atle Selberg \cite[pp.\,57-58]{S1} introduced the power function
$p_s:\CP\lrt \BC$ defined by
\begin{equation}
  p_s (Y):=\prod_{j=1}^{n} (\det Y_j)^{s_j},\quad Y\in \CP,
\end{equation}
where $Y_j\in \mathscr P_j\ (1\leq j\leq n)$ is the $j\times j$ upper left corner of $Y$.
Let
\begin{equation}
T_n:=\left\{\, t=
\begin{pmatrix}
           t_{11} & t_{12} & \cdots & t_{1n} \\
           0 & t_{22} & \cdots & t_{2n}\\
           0 & 0 & \ddots &  \vdots\\
           0 & 0 &   0  & t_{nn}
         \end{pmatrix}\in GL(n,\BR)\,\Big|\ t_{jj}>0,\ 1\leq j\leq n \ \right\}
\end{equation}
be the subgroup of $GL(n,\BR)$ consisting of upper triangular matrices. For $r=(r_1,\cdots,r_n)\in \BC^n$,
we define the group homomorphism $\tau_r : T_n\lrt \BC^{\times}$ by
\begin{equation}
  \tau_r (t):=\prod_{j=1}^{n} t_{jj}^{r_j},\quad t=(t_{ij})\in T_n.
\end{equation}
For $z=(z_1,\cdots,z_n)\in \BC^n$, we define the function $\phi_z: T_n\lrt \BC^{\times}$ by
\begin{equation}
  \phi_z (t):= \prod_{j=1}^{n} t_{jj}^{2 z_j+j-\frac{n+1}{2}},\quad t=(t_{ij})\in T_n.
\end{equation}
We note that $\phi_z (t)=p_s (I_n[t])$ for some $s\in\BC^n$.

\begin{proposition}
(1) For $s=(s_1,\cdots,s_n)\in \BC^n$, we put $r_j=2(s_j+\cdots + s_n),\ j=1,\cdots, n.$
Then we have
\begin{equation*}
  p_s (I_n [t])=\tau_r (t),\quad t\in T_n.
\end{equation*}
\vskip2mm
\noindent
(2) $p_s (Y[t])=p_s(I_n[t])\,p_s(Y)$ for any $Y\in \CP$ and $t\in T_n$.
\vskip 2mm
\noindent
(3) For any $D\in \BD (\CP)$, we have $Dp_s=Dp_s (I_n)\,p_s,$ i.e., $p_s$ is a common eigenfunction
of $\BD (\CP)$.
\end{proposition}
\begin{proof} The proof can be found in \cite[pp.\,39--40]{T}.\end{proof}

\vskip 3mm
Hans Maass \cite{M2} proved the following theorem.
\begin{theorem}
(1) Let $\delta_1,\delta_2,\cdots,\delta_n\in  \BD (\CP)$ be Maass-Selberg operators given by Formula (2.4). Then
\begin{equation*}
  \delta_j \phi_z = \lambda_j (z) \phi_z, \quad 1\leq j\leq n.
\end{equation*}
where $\lambda_j (z)$ is a symmetric polynomial in $z_1,\cdots, z_n$ of degree $j$ and having the following form:
\begin{equation*}
\lambda_j (z_1,\cdots, z_n)=z_1^j+\cdots+z_n^j + {\rm terms\ of\ lower\ degree}.
\end{equation*}
\noindent
(2) The effect of $D\in \BD (\CP)$ on power functions $p_s (Y)$ determines $D$ uniquely.
\end{theorem}
\begin{proof} The proof can be found in \cite[pp.\,70--76]{M2} or \cite[pp.\,44--48]{T}.\end{proof}

\vskip 3mm
A function $h:\CP\lrt \BC$ is said to be $\textsf{spherical}$ if $h$ satisfies
the following properties (2.9)--(2.11):
\begin{equation}
  h(Y[k])=h(\,^tk Yk)=h(Y)\quad {\rm for\ all}\ Y\in \CP\ {\rm and}\ k\in O(n,\BR).
\end{equation}
\begin{equation}
h\ {\rm is\ a\ common\ eigenfunction\ of} \ \BD (\CP).
\end{equation}
\begin{equation}
h(I_n)=1.
\end{equation}

\vskip 2mm
For the present being, we put $G=GL(n,\BR)$ and $K=O(n,\BR)$. For $s=(s_1,\cdots,s_n)\in \BC^n$,
we define the function
\begin{equation}
h_s (Y):=\int_{K} p_s(Y[k])\,dk, \quad Y\in \CP,
\end{equation}
where $dk$ is a normalized measure on $K$ so that $\int_K dk =1.$ It is easily seen that $h_s (Y)$ is a spherical function on $\CP.$ Selberg \cite[pp.\,53--57]{S1} proved that these $h_s (Y)$ are the only spherical functions on $\CP.$ If $f\in C_c (\CP)$, the $\textsf{Helgason-Fourier\ transform}$ of $f$ is defined to be the function ${\mathscr H}f:\BC^n\times K\lrt \BC$\,:
\begin{equation}
{\mathscr H}f (s,k):=\int_{\CP} f(Y)\,\overline{p_s(Y[k])}\,dv_n (Y), \quad
{\rm for\ all}\ (s,k)\in \BC^n\times K,
\end{equation}
where $p_s$ is the Selberg power function (see Formula (2.5)) and $dv_n (Y)$ is a $GL(n,\BR)$-invariant volume element on $\CP$ (see Formula (2.3)).

\begin{proposition}
(1) The spherical function $h$ on $\CP$ corresponding to the eigenvalues
$\lambda_1,\cdots,\lambda_n\in \BC$ with
\begin{equation*}
\delta_i h= \lambda_i h,\quad 1\leq i\leq n
\end{equation*}
is unique. Here $\delta_1,\delta_2\cdots,\delta_n$ are Maass-Selberg operators
on $\CP$ defined by Formula (2.4).
\vskip 2mm \noindent
(2) Let $f\in C_c (\CP)$ be a common eigenfunction of $\BD (\CP)$, i.e., $Df=\lambda_D f \,(\lambda_D\in \BC)$ for all $D\in \BD (\CP).$ Define $s\in \BC^n$ by
\begin{equation*}
Dp_s=\lambda_D\, p_s, \quad D\in \BD (\CP).
\end{equation*}
If $g\in C^{\infty}_c (K\backslash G/K)$ is a $K$-bi-invariant function on $G$ satisfying the condition $g(x)=g(x^{-1})$ for all $x\in G$, then
\begin{equation*}
f \star g = {\hat g}({\bar s}) f,
\end{equation*}
where $\star$ denotes the convolution operator and
\begin{equation*}
{\hat g}({\bar s}):=\int_{\CP} g(Y)\,\overline{p_{\bar s}(Y[k])}\,dv_n (Y).
\end{equation*}
Conversely, suppose that $f\in C^{\infty}(\CP)$ is a $K$-invariant eigenfunction of all convolution operators with $g\in C^{\infty}_c (\CP).$ Then $f$ is a common eigenfunction of $\BD (\CP)$.
\end{proposition}
\begin{proof} The proof can be found in \cite[pp.\,53--56]{S1} or \cite[pp.\,67--69]{T}.\end{proof}

\vskip 3mm
The fundamental domain ${\mathfrak R}_n$ for $GL(n,\BZ)\ba \CP$ which was
found by H. Minkowski \cite{Mi} is defined as a subset of $\CP$
consisting of $Y=(y_{ij})\in \CP$ satisfying the following
conditions (M.1)--(M.2)\ (cf.\,\cite[p.\,123]{M2}):

\vskip 0.1cm
(M.1)\ \  $aY\,^ta\geq
y_{kk}$\ \ for every $a=(a_i)\in\BZ^n$ in which $a_k,\cdots,a_n$
are relatively prime for \\
\indent \ \ \ \ \ \ \ \ \ $k=1,2,\cdots,n$.

\vskip 0.1cm
(M.2)\ \ \ $y_{k,k+1}\geq 0$ \ for $k=1,\cdots,n-1.$

\vskip 0.1cm
We say
that a point of $\Rg$ is {\it Minkowski reduced} or simply {\it
M}-{\it reduced}. $\Rg$ has the following properties (R1)--(R4)\,(cf.\,\cite[p.\,139]{M2}):

\vskip 0.1cm
(R1) \ For any $Y\in\CP,$ there exist a matrix $A\in
GL(n,\BZ)$ and $R\in\Rg$ such that
\\
\indent \ \ \ \ \ \ \ \ \
$Y=R[A]$. That is,
\begin{equation*}
  GL(n,\BZ)\circ \Rg=\CP.
\end{equation*}
\indent (R2)\ \ $\Rg$ is a convex cone through the origin bounded
by a finite number of hyperplanes.
\\
\indent \ \ \ \ \ \ \ \
$\Rg$ is closed in $\CP$.

\vskip 0.1cm
(R3) If $Y$ and $Y[A]$ lie in $\Rg$ for $A\in
GL(n,\BZ)$ with $A\neq \pm I_n,$ then $Y$ lies on the boundary
\\
\indent \ \ \ \ \ \ \
$\partial \Rg$ of $\Rg$. Moreover $\Rg\cap (\Rg [A])\neq
\emptyset$ for only finitely many $A\in GL(n,\BZ).$

\vskip 0.1cm (R4) If $Y=(y_{ij})$ is
an element of $\Rg$, then
\begin{equation*}
y_{11}\leq y_{22}\leq \cdots \leq y_{nn}\quad \text{and}\quad
|y_{ij}|<{\frac 12}y_{ii}\quad \text{for}\ 1\leq i< j\leq n.
\end{equation*}
\indent We refer to \cite[pp.\,123--124]{M2} for more details.

\vskip 0.3cm
$\CP$ parameterizes principally polarized real tori of dimension $n$ (cf.\,\cite{Y4}). The arithmetic quotient $GL(n,\BZ)\backslash \CP$ is the moduli space of isomorphism classes of principally polarized
real tori of dimension $n$.
According to (R2) we see that $\Rg$ is a semi-algebraic set with real analytic structure.
Unfortunately $GL(n,\BZ)\backslash \CP$ does not admit the structure of a real algebraic variety and does not admit a compactification which is defined over the rational number field $\BQ$\,(cf.\,\cite{Sil} or
\cite{GoT}).

\vskip 3mm
We give the definition of automorphic forms for $GL(n,\BZ)$ given by A. Terras\,
(cf. \cite[p.\,182]{T}).

%Definition 2.1
\begin{definition}
A real analytic function $f:\CP\lrt \BC$ is said to be a {\sf automorphic form} for $GL(n,\BZ)$
if it satisfies the following conditions {\rm (A1)}--{\rm (A3)}\,:
\vskip 2mm\noindent
\ {\rm (A1)} \,$f(\gamma Y\,{}^t\gamma)=f(Y)$ \ \ for all $Y\in \CP$ and all $\gamma\in GL(n,\BZ)$\,;
\vskip 2mm\noindent
\ {\rm (A2)} $f$ is an eigenfunction of all $D\in \BD(\CP)$\; i.e., $Df=\la_D f$ for some eigenvalue
$\la_D$;
\vskip 2mm\noindent
\ {\rm (A3)} $f$ has at most polynomial growth at infinity\,; i.e.,
$$|f(Y)|\leq C\,|p_s(Y)|\qquad {\rm for\ some}\ s\in\BC^n\ {\rm and}\ C>0.$$
\end{definition}
We set $\G=GL(n,\BZ)$. We denote by ${\bf A}(\G,\la)$ the space of all automorphic forms for $\G$
with a given eigenvalue system $\la$. An automorphic form $f$ in ${\bf A}(\G,\la)$ is called
a {\sf cusp form} for $\G$ if for any $k$ with $1\leq k\leq n-1,$ we have
\begin{equation}
  \int_{X\in T^{(k,n-k)}} f\left(
\,{}^{{}^{{}^{{}^\text{\scriptsize $t$}}}}\!\!\!\!\begin{pmatrix}
                                        I_k & x \\
                                        0 & I_{n-k}
                                      \end{pmatrix} Y
                                      \begin{pmatrix}
                                        I_k & x \\
                                        0 & I_{n-k}
                                      \end{pmatrix}\right)dx=0\quad {\rm for\ all}\ Y\in \CP.
\end{equation}
Here $T=\BR/\BZ$ denotes a circle, that is, a one-dimensional torus and $T^{(k,n-k)}$
denotes the set of all $k\times (n-k)$ matrices with entries in $T$. The condition (A3)
implies the vanishing of the constant terms in some Fourier expansions of $f(Y)$ as
a periodic function in the $x$-variable in partial Iwasawa coordinates.

\begin{remark}
Borel and Jacquet defined automorphic forms for a connected reductive group over $\BQ$
({\rm cf.}\,\cite[pp.\,199--200]{Bo} and \cite[pp.\,189--190]{B-J}). The definition given by
Borel and Jacquet is slightly different from Definition 2.1 given by Terras.
\end{remark}

\vskip 3mm
One of the motivations to study automorphic forms for $GL(n,\BZ)$ is the need to study various kinds
of $L$-functions with many gamma factors in their functional equations. Another motivation for
the study of automorphic forms for $GL(n,\BZ)$ is to develop the theory of harmonic analysis on
$L^2 (GL(n,\BZ)\backslash \CP)$ and $L^2 (GL(n,\BZ)\backslash GL(n,\BR))$ which involves
in the unitary representations of $GL(n,\BR)$.

\begin{remark}
Using the Grenier operator defined by Douglas Grenier ({\rm cf.}\,\cite{Gr2}), we can
define the notion of {\sf stable automorphic forms} for $GL(n,\BZ)$. We use
stable automorphic forms to study the geometric and arithmetic properties of the
universal family of principal polarized real tori. We refer to \cite{Y7} for more details.
\end{remark}

\end{section}

%%%%%%%%%%%%%%%%%%%%%%%%%%%%%%%%%%%%%%%%%%%%%%%%%%%%%%%%%%%%%%%%%%%%%%%%%%%%%%%%%%%%%%%%%%%%%%%%%%%%%%%%%%%%%%%%%%%%%%%%%%%%%%%%%%%%%%%%%%%%%%%%%%%%%%%%
%%%%%%%%%%%%%%%%%%%%%%%%%%%%%%%%%%%%%%%%%%%%%%%%%%%%%%%%%%%%%%%%%%%%%%%%%%%%%%%%%%%%%%%%%%%%%%%%%%%%%%%%%%%%%%%%%%%%%%%%%%%%%%%%%%%%%%%%%%%%%%%%%%%%%%%%
%%%%%%%%%%%%%%%%%%%%%%%%%%%%%%%%%%%%%%%%%%%%%%%%%%%%%%%%%%%%%%%%%%%%%%%%%%%%%%%%%%%%%%%%%%%%%%%%%%%%%%%%%%%%%%%%%%%%%%%%%%%%%%%%%%%%%%%%%%%%%%%%%%%%%%%%
%%%%%%%%%%%%%%%%%%%%%%%%%%%%%%%%%%%%%%%%%%%%%%%%%%%%%%%%%%%%%%%%%%%%%%%%%%%%%%%%%%%%%%%%%%%%%%%%%%%%%%%%%%%%%%%%%%%%%%%%%%%%%%%%%%%%%%%%%%%%%%%%%%%%%%%%
%%%%%%%%
%%%%%%%%
%%%%%%%%                   3. Automorphic Forms for $SL(n,\BZ)
%%%%%%%%
%%%%%%%%
%%%%%%%%%%%%%%%%%%%%%%%%%%%%%%%%%%%%%%%%%%%%%%%%%%%%%%%%%%%%%%%%%%%%%%%%%%%%%%%%%%%%%%%%%%%%%%%%%%%%%%%%%%%%%%%%%%%%%%%%%%%%%%%%%%%%%%%%%%%%%%%%%%%%%%%%
%%%%%%%%%%%%%%%%%%%%%%%%%%%%%%%%%%%%%%%%%%%%%%%%%%%%%%%%%%%%%%%%%%%%%%%%%%%%%%%%%%%%%%%%%%%%%%%%%%%%%%%%%%%%%%%%%%%%%%%%%%%%%%%%%%%%%%%%%%%%%%%%%%%%%%%%
%%%%%%%%%%%%%%%%%%%%%%%%%%%%%%%%%%%%%%%%%%%%%%%%%%%%%%%%%%%%%%%%%%%%%%%%%%%%%%%%%%%%%%%%%%%%%%%%%%%%%%%%%%%%%%%%%%%%%%%%%%%%%%%%%%%%%%%%%%%%%%%%%%%%%%%%
%%%%%%%%%%%%%%%%%%%%%%%%%%%%%%%%%%%%%%%%%%%%%%%%%%%%%%%%%%%%%%%%%%%%%%%%%%%%%%%%%%%%%%%%%%%%%%%%%%%%%%%%%%%%%%%%%%%%%%%%%%%%%%%%%%%%%%%%%%%%%%%%%%%%%%%%

\vskip 10mm
\begin{section}{{\bf Automorphic Forms for $SL(n,\BZ)$}}
\setcounter{equation}{0}

\vskip 3mm
Let
\begin{equation*}
  \SP:=\left\{ Y\in \BR^{(n,n)}\,|\ Y=\,^tY >0,\ \det (Y)=1\ \right\}.
\end{equation*}
be a symmetric space associated to $SL(n,\BR)$. Indeed, $SL(n,\BR)$ acts on $\SP$ transitively by
\begin{equation}
  g\circ Y= gY\,{}^tg,\qquad g\in SL(n,\BR),\ Y\in \SP.
\end{equation}
Thus $\SP$ is a smooth manifold diffeomorphic to the symmetric space $SL(n,\BR)/SO(n,\BR)$ through the bijective map
\begin{equation*}
  SL(n,\BR)/SO(n,\BR) \lrt \SP,\qquad g\cdot SO(n,\BR) \mapsto g\circ I_n=\,g\,{}^tg,\quad g\in SL(n,\BR).
\end{equation*}

\vskip 3mm
For $Y\in\SP$, we have a partial Iwasawa decomposition
\begin{equation}
  Y= \begin{pmatrix}
       v^{-1} & 0 \\
       0 & v^{1/(n-1)}W
      \end{pmatrix}
       \left[ \begin{pmatrix}
                1 & ^tx \\
                0 & I_{n-1}
              \end{pmatrix} \right] =\begin{pmatrix}
                v^{-1} & v^{-1}\,^tx \\
                v^{-1}x & v^{-1}x\,^tx + v^{1/(n-1)}W
              \end{pmatrix}
\end{equation}
where $v>0,\ x\in \BR^{(n-1,1)}$ and $W\in \mathfrak{P}_{n-1}.$ From now on, for brevity, we
write $Y=[v,x,W]$ instead of the decomposition (3.2). In these coordinates $Y=[v,x,W]$,
\begin{equation*}
  ds_Y^2 =\frac{n}{n-1}\, v^{-2} dv^2 + 2\,v^{-n/(n-1)} W^{-1}[dx] + ds_W^2
\end{equation*}
is a $SL(n,\BR)$-invariant metric on $\SP$, where $dx=\,^t(dx_1,\cdots, dx_{n-1})$ and
$ds_W^2$ is a $SL(n-1,\BR)$-invariant metric on $\mathfrak P_{n-1}$.
The Laplace operator $\Delta_n$ of $(\SP, ds_Y^2)$ is given by
\begin{equation*}
   \Delta_n=\frac{n-1}{n}\, v^{2} {{\partial^2}\over {\partial v^2}} - \frac{1}{n}
   {{\partial}\over {\partial v}} + v^{n/(n-1)}\, W \left[ {{\partial}\over {\partial x}} \right]+\Delta_{n-1}
\end{equation*}
inductively, where if $x=\,^t(x_1,\cdots,x_{n-1})\in \BR^{(n-1,1)}$,
$${{\partial}\over {\partial x}}=\,{}^{{}^{{}^{{}^\text{\scriptsize $t$}}}}\!\!\!
\left( {{\partial}\over {\partial x_1}},\cdots,
{{\partial}\over {\partial x_{n-1}}}  \right)$$
and $\Delta_{n-1}$ is the Laplace operator of $(\mathfrak P_{n-1}, ds^2_W)$.
\begin{equation*}
d\mu_n=v^{-(n+2)/2}\, dv\, dx\, d\mu_{n-1}
\end{equation*}
is a $SL(n,\BR)$-invariant volume element on $\SP$, where $dx=dx_1\cdots dx_{n-1}$ and
$d\mu_{n-1}$ is a $SL(n-1,\BR)$-invariant volume element on $\mathfrak P_{n-1}$.

\vskip 3mm
Following earlier work of Minkowski\,\cite{Mi}, Siegel \cite{Si} showed that the volume of the fundamental domain
$SL(n,\BZ)\backslash \SP$ is given as follows\,:
\begin{equation}
  {\rm Vol}(SL(n,\BZ)\backslash \SP)
%  ={\rm Vol}(SL(n,\BZ)\backslash SL(n,\BR)/SO(n,\BR))
  =\int_{SL(n,\BZ)\backslash \SP}d\mu_n=
  n\,2^{n-1}\prod_{k=2}^{n}{{\zeta (k)}\over {{\rm Vol}(S^{k-1})}},
\end{equation}
where
\begin{equation*}
{\rm Vol}(S^{k-1})={{2\,(\sqrt{\pi})^k}\over {\Gamma(k/2)}}
\end{equation*}
denotes the volume of the $(k-1)$-dimensional sphere $S^{k-1}$, $\Gamma (x)$ denotes the usual Gamma function and $\zeta (k)=\sum_{m=1}^{\infty} m^{-k}$ denotes the Riemann zeta function. The proof of (3.3) can be found in \cite[pp.\,27-37]{Go} and \cite{Ga}.

\vskip 3mm
Let $\BD(\mathfrak{P}_n)$ be the algebra of all differential operators on $\mathfrak{P}_n$
invariant under the action (3.1) of $SL(n,\BR)$. It is well known (cf.\,\cite{HC1,HC2,He})
that $\BD(\mathfrak{P}_n)$ is commutative and is isomorphic to the polynomial algebra $\BC [x_1,x_2,\cdots,x_{n-1}]$ with $n$ indeterminates $x_1,x_2,\cdots,x_{n-1}$.
We observe that $n-1$ is the rank of $SL(n,\BR)$, i.e., the rank of the symmetric space
$SL(n,\BR)/SO(n,\BR)$.
In \cite{BCH}, using the Maass-Selberg
operators $\delta_1,\delta_2,\cdots,\delta_n$, Brennecken, Ciardo and Hilgert found
explicit generators $E_1,E_2,\cdots,E_{n-1}$ of $\BD(\mathfrak{P}_n)$. Obviously
$E_1,E_2,\cdots,E_{n-1}$ are algebraically independent.

\vskip 3mm
If we repeat this partial decomposition process for $Y\in \mathfrak{P}_n$, we get the Iwasawa decomposition
\begin{equation}
  Y=y^{-1} {\rm{diag}}
  \left( 1,y_1^2,(y_1y_2)^2,\cdots,(y_1y_2\cdots y_{n-1})^2 \right)
  \left[ \begin{pmatrix}
           1 & x_{12} & \cdots & x_{1n} \\
           0 & 1 & \cdots & x_{2n}\\
           0 & 0 & \ddots &  \vdots\\
           0 & 0 &   0  & 1
         \end{pmatrix} \right],
\end{equation}
where $y>0,\ y_j\in \BR \,(1\leq j\leq n-1)$ and $x_{ij}\in\BR\, (1\leq i < j\leq n).$ Here
$y=y_1^{2(n-1)}\cdots y_{n-1}^2$ and ${\rm{diag}} (a_1,\cdots,a_n)$ denotes the $n\times n$ diagonal matrix with diagonal entries $a_1,\cdots,a_n$. In this case we denote $Y=(y_1,\cdots,y_{n-1},x_{12},\cdots,x_{n-1,n}).$

\vskip 5mm
Define $\Gamma_n=GL(n,\BZ)/\{ \pm I_n \}.$ We observe that $\G_n=SL(n,\BZ)/\{ \pm I_n \}$ if $n$ is even, and $\G_n=SL(n,\BZ)$ if $n$ is odd.
An \textsf{automorphic form} for $\G_n$ is defined to be a real analytic function $f\!:\!\SP\lrt \BC$ satisfying the following conditions (AF1)--(AF3)\,:
\vskip 2mm
(AF1) $f$ is an eigenfunction for all
$SL(n,\BR)$-invariant\ differential\ operators on $\SP$.
\vskip 2mm
(AF2) \ \ $f(\g Y\,{}^t\g)=f(Y)\quad {\rm for\ all}\ \g\in \G_n
\ {\rm and}\ Y\in \SP.$
\vskip 2mm
(AF3)\ There\ exist\ a constant  $C>0$ and $s\in \BC^{n-1}$
with  $s=(s_1,\cdots,s_{n-1})$   \\
\indent \ \ \ \ \ \ \ \ \ such that
$|f(Y)| \leq C\,|p_{-s}(Y)|$ as the upper left determinants
$\det Y_j\lrt \infty,$ \\
\indent \ \ \ \ \ \ \ \ \ $j=1,2,\cdots, n-1$, where
\begin{equation*}
  p_{-s}(Y):=\prod_{j=1}^{n-1} (\det Y_j)^{-s_j}
\end{equation*}
\indent \ \ \ \ \ \ \ \ \
is the Selberg's power function\,(cf.\,\cite{S1, T}).

\vskip 3mm
We denote by ${\bf A}(\G_n)$ the space of all automorphic forms for $\G_n.$ A \textsf{cusp form}
$f\in {\bf A}(\G_n)$ is an automorphic form for $\G_n$ satisfying the following conditions\,:
\begin{equation*}
\int_{X\in (\BR/\BZ)^{(j,n-j)}}
f \left( Y\left[ \begin{pmatrix}
                   I_j & X \\
                   0 & I_{n-j}
                 \end{pmatrix}\right]\right)dX=0,
                 \quad 1\leq j\leq n-1.
\end{equation*}
Here $(\BR/\BZ)^{(j,n-j)}$ denotes the set of all $j\times (n-j)$ matrices with entries
in the one-dimensional real torus $\BR/\BZ$.
We denote by ${\bf A}_0(\G_n)$ the space of all cusp forms for $\G_n.$

\vskip 3mm
For $s=(s_1,\cdots,s_{n-1})\in \BC^{n-1},$ we now consider the following Eisenstein series
\begin{equation}
  E_n(s,Y):=\sum_{\gamma\in \G_n/\G_{\star}} p_{-s} ( Y[\gamma]),\qquad  Y\in\SP,
\end{equation}
where $\G_{\star}$ is a subgroup of $\G_n$ consisting of all upper triangular matrices.
$E_n(s,Y)$ is a special type of the more general Eisenstein series introduced by Atle Selberg
\cite{S1}. It is known that the series converges for ${\rm Re} (s_j)>1,\ j=1,2,\cdots, n-1,$ and has analytic continuation for each of the variables $s_1,\cdots,s_{n-1}$ (cf.\,\cite{S2, T}). It is seen that $E_n(s,Y)$ is a common eigenfunction of all invariant differential operators in $\BD (\CP)$.
Its corresponding eigenvalue of the Laplace operator $\Delta_n$ is given by
\begin{eqnarray*}
  \lambda &=& \frac{n-1}{n} (s_1+\xi_1)\left( s_1-1+\xi_1 - \frac{1}{n-1} \right) \\
   & & + \frac{n-2}{n-1} (s_2+\xi_2)\left( s_2-1+\xi_2 - \frac{1}{n-2} \right)+\cdots +
   \frac{1}{2} s_{n-1}(s_{n-1}-2),
\end{eqnarray*}
where $\xi_{n-1}=0$ and
\begin{equation}
  \xi_j=\frac{1}{n-j} \sum_{k=j+1}^{n-1} (n-k) s_k,\qquad j=1,2,\cdots,n-2.
\end{equation}

\vskip 3mm
Let $f\in {\bf A}(\G_n)$ be an automorphic form. Since $f(Y)$ is invariant under the action of
the subgroup
$\left\{ \begin{pmatrix}
   1 & a \\
   0 & I_{n-1}
 \end{pmatrix} \ \Big|\ a\in \BZ^{(1,n-1)}\,\right\}$ of $\G_n$, we have the Fourier expansion
\begin{equation}
  f(Y)=\sum_{N\in \BZ^{(n-1,1)}} a_N (v,W) e^{2\pi i\, ^txN},
\end{equation}
where
$Y=[v,x,W]\in \SP$ and
\begin{equation*}
  a_N (v,W)=\int_{0}^{1}\int_{0}^{1}\cdots \int_{0}^{1} f([v,x,W])\, e^{-2\pi i\, ^txN} dx.
\end{equation*}

\vskip 2mm
For $s\in \BC^{n-1}$ and $A,B\in \SP$, we define the $K$-Bessel function
\begin{equation}
K_n (s\,|\,A,B):=\int_{\CP} p_s(Y) \,e^{{\rm Tr} (AY+BY^{-1})}dv_n,
\end{equation}
where $dv_n$ is the $GL(n,\BR)$-invariant volume element on $\CP$\,(see Formula (2.3)).

\vskip 3mm
Let $f\in {\bf A}(\G_n)$ be an automorphic form for $\G_n$. Thus we have $n$ differential equations
$\delta_j f=\lambda_j f\,(1\leq j\leq n).$ Here $\delta_1,\delta_2,\cdots,\delta_n$ are
$GL(n,\BR)$-invariant Maass-Selberg operators defined by (2.4). We can find $s=(s_j)\in \BC^{n-1}$
satisfying the various relations determined by the
$\lambda_j\,(1\leq j\leq n)$. D. Grenier \cite[Theorem 1, pp.\,469-471]{Gr2} proved that $f$ has the following Fourier expansion
\begin{eqnarray}
  f(Y)\! &=&\! f([v,x,W]) \nonumber\\
  \! &=&\! a_0(v,W)+\sum_{0\neq m\in \BZ^{n-1}} \sum_{\gamma\in \G_{n-1}/P} a_m (v,W) v^{(n-1)/2}\\
  \! & & \!\times\, K_{n-1}({\hat s}\,|\,v^{1/(n-1)}\,^t\gamma W\gamma,\pi^2 v\, m\,^tm)\,e^{2\pi i\, ^tx\gamma m},\nonumber
\end{eqnarray}
where ${\hat s}=\left( s_1-\frac{1}{2},s_2,\cdots,s_{n-1} \right)$ and $P$ denotes the parabolic subgroup of $\G_{n-1}$ consisting of the form
$\begin{pmatrix}
  \pm 1 & b \\
  0 & d
\end{pmatrix}$ with $d\in \G_{n-2}.$

\vskip 3mm
D. Grenier \cite{Gr1, Gr3} found a fundamental domain ${\mathfrak F}_n$ for $\G_n$ in $\SP$.
The fundamental domain ${\mathfrak F}_n$ is precisely the set of all $Y=[v,x,W]\in \SP$ satisfying
the following conditions (F1)--(F3)\,:
\vskip 1mm\noindent
(F1) $(a+\,^txc)^2+v^{n/(n-1)} W[c]\geq 1$ for all
$\begin{pmatrix}
  a & ^t b \\
  c & d
\end{pmatrix}\in \G_n$ with $a\in \BZ,\ b,c\in \BZ^{(n-1,1)}$ \\
\indent \indent and $d\in \BZ^{(n-1,n-1)}.$

\vskip 2mm\noindent
(F2) $W\in {\mathfrak F}_{n-1}.$

\vskip 2mm\noindent
(F3)\ $0\leq x_1 \leq \frac{1}{2},\ \, |x_j|\leq 2$ for $2\leq j\leq n-2.$ Here
$x=\,^t (x_1,\cdots,x_{n-1})\in \BR^{(n-1,1)}.$

\vskip 3mm
For a positive real number $t>0$, we define the $\textsf{Siegel\ set} \ {\mathscr S}_{t,1/2}$ by
\begin{equation*}
  {\mathscr S}_{t,1/2}:=\left\{ Y\in \SP\,\big|\ y_i\geq t^{-1/2}\,\,(1\leq i\leq n-1),
  \ |x_{ij}|\leq \frac{1}{2}\,\,(1\leq i < j \leq n)\,\right\}.
\end{equation*}
Here we used the coordinates on $\SP$ given in Formula (3.2).

\vskip 2mm
Grenier \cite{Gr3} proved the following theorems\,:
\begin{theorem}
  Let
\begin{equation*}
 {\mathfrak F}_n^{\sharp}:=\bigcup_{\gamma\in \mathfrak{D}_n} {\mathfrak F}_n [\gamma]\subset \SP,
\end{equation*}
where $\mathfrak{D}_n$ is the subgroup of $\G_n$ consisting of diagonal matrices
${\rm diag}(\pm 1,\cdots,\pm 1).$ Then
\begin{equation*}
{\mathscr S}_{1,1/2} \subset {\mathfrak F}_n^{\sharp} \subset {\mathscr S}_{4/3,1/2}.
\end{equation*}
\end{theorem}
\noindent
{\it Proof.} See Theorem 1 in \cite[pp.\,58--59]{Gr3}. \hfill $\square$

\begin{theorem}
  Let
\begin{equation*}
 {\mathfrak F}_n^*:={\mathfrak F}_n\cup {\mathfrak F}_{n-1} \cup \cdots \cup {\mathfrak F}_2
 \cup {\mathfrak F}_1
\end{equation*}
and
\begin{equation*}
V_n^*:=V_n \cup V_{n-1} \cup \cdots \cup V_1 \cup V_0,\quad
V_i:=\G_i\backslash \mathfrak{P}_i \,(1\leq i\leq n),\ V_0:=\{ \infty\}.
\end{equation*}
Then ${\mathfrak F}_n^*$ is a compact Hausdorff space whose topology is induced by the closure of
${\mathfrak F}_n$ in $\SP.$ And $V_n^*$ is a compact Hausdorff space called
the {\sf Satake\ compactification} of $V_n$.
\end{theorem}
\begin{proof} See Theorem 3 in \cite[pp.\,62--65]{Gr3}. \end{proof}

\begin{remark}
In \cite{BoJ1,BoJ2}, Borel and Ji constructed the geodesic compactification, the standard
compactification and a maximal Satake compactification of the locally symmetric space
$SL(n,\BZ)\ba SL(n,\BR)/SO(n,\BR).$
\end{remark}

\vskip 3mm
Now we give another coordinate system on $SL(n,\BR)/SO(n,\BR)$ given by
Goldfeld\,(cf.\,\cite[Definition 1.2.3, p.\,10]{Go}).

\begin{definition}
For any positive integer $n\geq 2$. we define ${\mathfrak H}_n$ to be the set of all
$n\times n$ real matrices of the form $z=x\cdot y$, where
$$x=\begin{pmatrix}
      1 & x_{12} & x_{13} &\cdots & x_{1n} \\
      0 & 1 &  x_{22} & \cdots & x_{2n} \\
      0 & 0 & \ddots & \vdots & \vdots  \\
      {0} & {0} & {0} & {0} & 1
    \end{pmatrix}$$
and
$$y={\rm {diag}}(y_1y_2\cdots y_{n-1},y_1y_2\cdots y_{n-2},\cdots,y_1,1)$$
with $x_{ij}\in \BR$ for $1\leq i < j \leq n$ and $y_i\geq 0 $ for $1\leq i\leq n-1.$
\end{definition}
We can show that ${\mathfrak H}_n$ is diffeomorphic to ${\mathfrak P}_n$. In fact, we have
the Iwasawa decomposition
\begin{equation*}
  GL(n,\BR)={\mathfrak H}_n\cdot O(n)\cdot Z_n,
\end{equation*}
where $Z_n(\cong \BR^{\times})$ is the center of $GL(n,\BR)$
\,(cf.\,\cite[Proposition 1.2.6, pp.\,11--12]{Go}).
Here
\begin{equation*}
  O(n):=O(n,\BR)=\{ k\in GL(n,\BR)|\ {}^tk k=k\,{}^tk=I_n\,\}
\end{equation*}
denotes the real orthogonal group of degree $n$.
We see easily that
\begin{equation*}
  {\mathfrak H}_n \cong GL(n,\BR)/(O(n)\cdot \BR^{\times}),
\end{equation*}
where $\cong$ denotes the diffeomorphism.
\vskip 2mm
It is seen that $GL(n,\BR)$ acts on ${\mathfrak H}_n$ by left translation
\,(cf.\,\cite[Proposition 1.2.10, p.\,14]{Go}).
We put
\begin{equation*}
  {\mathfrak X}_n:=SL(n,\BZ)\backslash SL(n,\BR)/SO(n) \cong SL(n,\BZ) \backslash
  GL(n,\BR)/(O(n)\cdot \BR^{\times}),
\end{equation*}
where $SO(n):=SO(n,\BR)=SL(n,\BR)\cap O(n).$
Therefore we obtain the following isomorphism
\begin{equation*}
  {\mathfrak X}_n\cong SL(n,\BZ)\backslash {\mathfrak H}_n.
\end{equation*}

\begin{proposition}
Let $n\geq 2$  and $\G^n=SL(n,\BZ).$ Using the coordinates of Definition 3.1, we put
\begin{equation*}
  d^*x=\prod_{1\leq i<j\leq n} dx_{ij}\quad {\rm and}\quad
  dy^*=\prod_{k=1}^{n-1} y_k^{-n(n-k)-1}dy_k.
\end{equation*}
Then
\begin{equation*}
  d^*z=d^*x\cdot d^*y
\end{equation*}
is the left $\G^n$-invariant volume element on
${\mathfrak P}_n\cong {\mathfrak H}_n.$
\end{proposition}
\begin{proof}
The proof can be found in \cite[Proposition 1.5.3, pp.\,25--26]{Go}.
\end{proof}

\begin{remark}
According to Formula (3.3) and Proposition 3.1, we see that
\begin{equation*}
  {\rm Vol}(\G^n\backslash {\mathfrak H}_n)=\int_{\G^n\backslash {\mathfrak H}_n}d^*z=
 \int_{\G^n\backslash {\mathfrak P}_n}d\mu_n = {\rm Vol}(\G^n\backslash {\mathfrak P}_n).
\end{equation*}

\end{remark}

\vskip 3mm
For any $\nu=(\nu_1,\nu_2,\cdots,\nu_{n-1}),$ we define the function
$I_\nu:{\mathfrak H}_n\lrt \BC$ by
\begin{equation}
%\label{}
  I_\nu (z):=\prod_{i=1}^{n-1} \prod_{i=1}^{n-1}y_i^{b_{ij}\nu_j}
  \quad {\rm for\ all}\ z=x\cdot y \in {\mathfrak H}_n,
\end{equation}
where
$$  b_{ij}:=\begin{cases}
            \ \ \ ij, & {\rm if }\ \, i+j\leq n \\
            (n-i)(n-j), & {\rm if}\ \, i+j\geq n.
          \end{cases}
$$
Then we see that $I_\nu(z)$ is an eigenfunction of ${\mathbb D}({\mathfrak H}_n).$
Here ${\mathbb D}({\mathfrak H}_n)$ denotes the algebra of all $SL(n,\BR)$-invariant
differential operators on ${\mathbb D}({\mathfrak H}_n)$.
Let us write
\begin{equation}\label{Formula (3.11)}
  DI_\nu(z)=\lambda_D\cdot I_\nu(z)\quad {\rm for\ every}\ D\in {\mathbb D}({\mathfrak H}_n).
\end{equation}
Since
$$\lambda_{D_1D_2}=\lambda_{D_1}\lambda_{D_2}\qquad {\rm for\ all}\ D_1,D_2\in
{\mathbb D}({\mathfrak H}_n).$$
The function $\lambda_D$ (viewed as a function of $D$) is a character of
${\mathbb D}({\mathfrak H}_n)$ which is called the {\sf Harish}-{\sf Chandra\ character}.

\vskip 3mm
Following Goldfeld\,(cf.\,\cite[Definition 5.1.3, pp.\,115--116]{Go}), the notion of
a Maass form is defined in the following way.
\begin{definition}
Let $n\geq 2$ and $\G^n=SL(n,\BZ).$ For any $\nu=(\nu_1,\nu_2,\cdots,\nu_{n-1})\in \BC^{n-1},$
a smooth function $f:\G^n\ba {\mathfrak H}_n\lrt \BC$ is said to be a {\sf Maass\ form} for $\G^n$ of
type $\nu$ if it satisfies the following conditions {\rm (M1)}--{\rm (M3)}\,:
\vskip 2mm\noindent
${\rm (M1)}\ \ f(\gamma\cdot z)=f(z)\quad {\rm for\ all}\ \gamma\in \G^n\ {\rm and}\ z\in {\mathfrak H}_n.$
\vskip 2mm\noindent
${\rm (M2)}\ \ Df(z)=\lambda_D f(z)\quad {\rm for\ all}\ D\in {\mathbb D}({\mathfrak H}_n)$
\ {\rm with\ eigenvalue}\ $\la_D$\ {\rm given\ by}\ {\rm (3.11)}.
\vskip 2mm\noindent
${\rm (M3)}\ \ \int_{(\G^n\cap U)\ba U} f(uz)du=0$
\ \ for all upper triangular groups $U$ of the form
$$U=\left\{
\begin{pmatrix}
  I_{r_1} & * & * & * \\
  0 & I_{r_2} & * & * \\
  0 & 0 & \ddots & * \\
  0 & 0 & 0 & I_{r_b}
\end{pmatrix}
\right\}$$
with $r_1+r_2+\cdots+r_b=n.$ Here $I_r$ denotes the $r\times r$ identity matrix and
$*$ denotes arbitrary real matrices.
\end{definition}
\begin{remark}
In his book \cite{Go}, Dorian Goldfeld studied
Whittaker functions associated with Maass forms, Hecke operators for $\G^n$,
the Godement-Jacquet $L$-function for $\G^n$, Eisenstein series for $\G^n$ and
Poincar{\'e} series for $\G^n$.
\end{remark}

\begin{remark}
Using the Grenier operator defined by Douglas Grenier ({\rm cf.}\,\cite{Gr2}), we can
define the notion of {\sf stable automorphic forms} for $SL(n,\BZ)$. We use
these stable automorphic forms to study the geometric and arithmetic properties of the
universal family of {\sf special} principal polarized real tori. We refer to \cite{Y6}
for more details.
\end{remark}

\end{section}

%%%%%%%%%%%%%%%%%%%%%%%%%%%%%%%%%%%%%%%%%%%%%%%%%%%%%%%%%%%%%%%%%%%%%%%%%%%%%%%%%%%%%%%%%%%%%%%%%%%%%%%%%%%%%%%%%%%%%%%%%%%%%%%%%%%%%%%%%%%%%%%%%%%%%%%%
%%%%%%%%%%%%%%%%%%%%%%%%%%%%%%%%%%%%%%%%%%%%%%%%%%%%%%%%%%%%%%%%%%%%%%%%%%%%%%%%%%%%%%%%%%%%%%%%%%%%%%%%%%%%%%%%%%%%%%%%%%%%%%%%%%%%%%%%%%%%%%%%%%%%%%%%
%%%%%%%%%%%%%%%%%%%%%%%%%%%%%%%%%%%%%%%%%%%%%%%%%%%%%%%%%%%%%%%%%%%%%%%%%%%%%%%%%%%%%%%%%%%%%%%%%%%%%%%%%%%%%%%%%%%%%%%%%%%%%%%%%%%%%%%%%%%%%%%%%%%%%%%%
%%%%%%%%%%%%%%%%%%%%%%%%%%%%%%%%%%%%%%%%%%%%%%%%%%%%%%%%%%%%%%%%%%%%%%%%%%%%%%%%%%%%%%%%%%%%%%%%%%%%%%%%%%%%%%%%%%%%%%%%%%%%%%%%%%%%%%%%%%%%%%%%%%%%%%%%
%%%%%%%%
%%%%%%%%
%%%%%%%%                   4. Automorphic Forms for GL(n,\BZ)\ltimes \BZ^{(m,n)}
%%%%%%%%
%%%%%%%%
%%%%%%%%%%%%%%%%%%%%%%%%%%%%%%%%%%%%%%%%%%%%%%%%%%%%%%%%%%%%%%%%%%%%%%%%%%%%%%%%%%%%%%%%%%%%%%%%%%%%%%%%%%%%%%%%%%%%%%%%%%%%%%%%%%%%%%%%%%%%%%%%%%%%%%%%
%%%%%%%%%%%%%%%%%%%%%%%%%%%%%%%%%%%%%%%%%%%%%%%%%%%%%%%%%%%%%%%%%%%%%%%%%%%%%%%%%%%%%%%%%%%%%%%%%%%%%%%%%%%%%%%%%%%%%%%%%%%%%%%%%%%%%%%%%%%%%%%%%%%%%%%%
%%%%%%%%%%%%%%%%%%%%%%%%%%%%%%%%%%%%%%%%%%%%%%%%%%%%%%%%%%%%%%%%%%%%%%%%%%%%%%%%%%%%%%%%%%%%%%%%%%%%%%%%%%%%%%%%%%%%%%%%%%%%%%%%%%%%%%%%%%%%%%%%%%%%%%%%
%%%%%%%%%%%%%%%%%%%%%%%%%%%%%%%%%%%%%%%%%%%%%%%%%%%%%%%%%%%%%%%%%%%%%%%%%%%%%%%%%%%%%%%%%%%%%%%%%%%%%%%%%%%%%%%%%%%%%%%%%%%%%%%%%%%%%%%%%%%%%%%%%%%%%%%%
\vskip 10mm
\begin{section}{{\bf Automorphic Forms for $GL(n,\BZ)\ltimes \BZ^{(m,n)}$}}
\setcounter{equation}{0}
\vskip 2mm
For two positive integers $m,\,n\in\BZ^+$, we introduce the new group
$$ GL_{n,m}(\BR):=GL(n,\BR)\ltimes \BR^{(m,n)}$$
which is the semidirect product of $GL(n,\BR)$ and the additive group $\BR^{(m,n)}.$
$GL_{n,m}(\BR)$ is endowed with the following with multiplication law
\begin{equation}\label{Formula (5.1)}
  (A,a)\circ (B,b):=(AB,a\,{}^tB^{-1}+b)
\end{equation}
for all $A,B\in GL(n,\BR)$ and $a,b\in\BR^{(m,n)}$.
We let
$$GL_{n,m}(\BZ):=GL(n,\BZ)\ltimes \BZ^{(m,n)}$$
be a discrete subgroup of $GL_{n,m}(\BR)$.
\vskip 2mm
For any two positive integers $m,\,n\in\BZ^+$, we let
\begin{equation*}
  {\mathscr P}_{n,m}:={\mathscr P}_n\times \BR^{(m,n)}
\end{equation*}
be the so-called {\sf Minkowski-Euclid space} of type $(n,m)$\ (cf.\,\cite{Y2,Y3}).
Then $GL_{n,m}(\BR)$ acts on ${\mathscr P}_{n,m}$
naturally and transitively by
\begin{equation}\label{Formula (4.2)}
  (A,a)\cdot (Y,V):=(AY\,{}^t\!A,(V+a)\,{}^t\!A)
\end{equation}
for all $(A,a)\in GL_{n,m}(\BR)$ and $(Y,V)\in {\mathscr P}_{n,m}$ with
$Y\in \mathscr{P}_n$ and $V\in \BR^{(m,n)}$. It is easily seen that
the stabilizer of the action (4.2) at $(I_n,0)$ is given by
$$K_{n,m}(\BR):=\{\, (k,0)\in GL_{n,m}(\BR)\,|\ k\in O(n,\BR)\,\}\cong O(n,\BR).$$
Thus the non-symmetric homogeneous space $GL_{n,m}(\BR)/O(n,\BR)$ is diffeomorphic to
the Minkowski-Euclid space ${\mathscr P}_{n,m}$ by the following correspondence
$$ (A,a)\cdot O(n,\BR)\longmapsto (A,a)\cdot (I_n,0)=(A\,{}^t\!A,a\,{}^t\!A)\quad
{\rm for\ all}\ (A,a)\in GL_{n,m}(\BR).$$

\vskip 3mm
For a variable $(Y,V)\in {\mathscr P}_{n,m}$ with
$Y\in \CP$ and $V\in \Rmn$, we put
$$Y=(y_{ij})\ \text{with}\ y_{ij}=y_{ji},\ \
V=(v_{kl}),$$   $$ dY=(dy_{ij}),\ \ dV=(dv_{kl}),$$
$$[dY]=\bigwedge_{i\leq j}dy_{ij},\ \ \ \ \
[dV]=\bigwedge_{k,l}dv_{kl},$$ and
$$\Yd=\left( { {1+\delta_{ij}}\over 2 } {
{\partial}\over {\partial y_{ij} }}\right),\ \ \ \Vd=\left({
{\partial}\over {\partial v_{kl}} } \right),$$ where $1\leq
i,j,l\leq n$ and $1\leq k\leq m.$

\vskip 0.2cm For a fixed element $(g,\lambda)\in \Gnm (\BR),$ we write
\begin{equation*}
(Y_{\star},V_{\star})=(g,\lambda)\cdot
(Y,V)=\big(g\,Y\,^tg,(V+\lambda)\,^tg\big),
\end{equation*}

\noindent where $(Y,V)\in {\mathscr P}_{n,m}$. Then we get

\begin{equation}
Y_{\star}= g\,Y\,\,^tg,\quad V_{\star}=(V+\lambda)\,^tg
\end{equation}
\noindent and

\begin{equation}
{ {\partial\,\, }\over {\partial Y_{\star}} }=
\,^tg^{-1}\,\Yd\,g^{-1},\quad {{\partial\,\,}\over {\partial
V_{\star}} } ={{\partial\,\,}\over {\partial V}}\,g^{-1}.
\end{equation}

\newcommand\fgt{\fg_{\star}}
\newcommand\Adt{$ \textrm{Ad}_{\star}$ }
\newcommand\fpt{\fp_{\star}}
\newcommand\Yst{Y_{\star}}
\newcommand\Vst{V_{\star}}
\newcommand\dYt{ { {\partial}\over {\partial \Yst} } }
\newcommand\dVt{ { {\partial}\over {\partial \Vst} } }

\begin{lemma}
For all two positive real numbers $A$ and $B$,
the following metric $ds^2_{n,m;A,B}$ on ${\mathscr P}_{n,m}$ defined by
\begin{equation}
ds^2_{n,m;A,B}=A\cdot {\rm Tr}\,(Y^{-1}dY\,Y^{-1}dY)\,+B\cdot {\rm Tr}\,(Y^{-1}\,^t(dV)\,dV)
\end{equation}
is a Riemannian metric on ${\mathscr P}_{n,m}$ which is invariant under the action (4.2)
of $\Gnm (\BR)$. The Laplace operator $\Delta_{n,m;A,B}$ of
$({\mathscr P}_{n,m},\,ds^2_{n,m;A,B})$ is given by
$$\Delta_{n,m;A,B}={\frac 1A}\cdot {\rm Tr} \left(\left( Y {{\partial}\over {\partial
Y}}\right)^2\right) -{\frac m{2\,A}}\cdot {\rm Tr}\left( Y {{\partial}\over
{\partial Y}}\right)\,+ \,{\frac 1B}\,\sum_{k\leq p} \left( \left({{\partial}\over
{\partial V}}\right) Y
{}^{{}^{{}^{{}^\text{\scriptsize $t$}}}}\!\!\!
\left({{\partial}\over {\partial
V}}\right)\right)_{kp}.$$ Moreover $\Delta_{n,m;A,B}$ is a
differential operator of order 2 which is invariant under the
action (4.2) of $\Gnm (\BR).$
Here for a matrix $M$ we denote by $M_{kp}$ the $(k,p)$-entry of $M$.
\end{lemma}

\begin{proof}
The proof can be found in \cite{Y4}.
\end{proof}

\vskip 0.5cm
\begin{lemma} The following volume element $dv_{n,m}(Y,V)$ on ${\mathscr P}_{n,m}$ defined by
\begin{equation}
dv_{n,m}(Y,V)=(\det\,Y)^{-{ {n+m+1}\over 2} }[dY][dV]
\end{equation} is invariant under the action (4.2)
of $\Gnm (\BR)$.
\end{lemma}

\begin{proof}
The proof can be found in \cite{Y4}.
\end{proof}

\vskip 3mm
We let
$$ \G_n:=GL(n,\BZ)\qquad {\rm and}\qquad \G_{n,m}:=\G_n \ltimes \BZ^{(m,n)}$$
be the discrete arithmetic subgroup of $GL(n,\BR)$ and $GL_{n,m}(\BR)$ respectively. We let
$${\mathfrak Y}_n:=\G_n\ba {\mathscr P}_n$$
be the Minkowski moduli space. If $Y_1,Y_2\in {\mathscr P}_n$ such that
$Y_2=\gamma Y_1\,{}^t\gamma =Y_1[\,{}^t\gamma]$ for some $\g \in \G_n$, then the real torus
$T(Y_1):=\BR^{(m,n)}/L(Y_1)$ is diffeomorphic to the torus $T(Y_2):=\BR^{(m,n)}/L(Y_2)$,
where $L(Y_1)=\BZ^{(m,n)}Y_1$ and $L(Y_2)=\BZ^{(m,n)}Y_2$ are lattices in the Euclidean space
$\BR^{(m,n)}$. Let
\begin{equation}\label{Formula (4.3)}
  \mathfrak{Y}_{n,m}:=\G_{n,m}\ba \mathscr{P}_{n,m}
\end{equation}
be the universal family of principally polarized real tori of dimension $mn$.
Roughly we say that $\mathfrak{Y}_{n,m}$ is a fibre bundle over
$\mathfrak{Y}_n$ with fibres given by principally polarized real tori.
We refer to \cite{Y3,Y4} for related topics about $\mathfrak{Y}_n$ and $\mathfrak{Y}_{n,m}.$

\vskip 3mm
For any two positive integers $r,s\in \BZ^+$ with $r<s$ and a fixed positive integer
$m\in\BZ^+,$ we define
$$ \xi_{r,s}:GL_{r,m}(\BR)\lrt GL_{s,m}(\BR)$$
by
\begin{equation*}
  \xi_{r,s}((A,0)):=\left( \begin{pmatrix}
                       A & 0 \\
                       0 & I_{s-r}
                     \end{pmatrix}, (a,0)\right),
\end{equation*}
where $A\in GL(r,\BR),\ a\in \BR^{(m,r)}$ and $(a,0)\in \BR^{(m,s)}.$
\vskip 3mm
For a fixed positive integer $m\in\BZ^+,$ we let
$$
GL_{\infty,m}:=\,\lim_{\begin{subarray}{c} \longrightarrow\\ ^r \end{subarray}}
GL_{r,m}(\BR), \quad
K_{\infty,m}:=\,\lim_{\begin{subarray}{c} \longrightarrow\\ ^r \end{subarray}}
K_{r,m}(\BR)\quad {\rm and}\quad
\G_{\infty,m}:=\,\lim_{\begin{subarray}{c} \longrightarrow\\ ^r \end{subarray}}
\G_{r,m}
$$
be the inductive limits of the directed systems
$(GL_{r,m}(\BR),\xi_{r,s}),\ (K_{r,m}(\BR),\xi_{r,s})$ and $(\G_{r,m}(\BR),\xi_{r,s})$
respectively.
\vskip 3mm
For any two positive integers $r,s\in \BZ^+$ with $r<s$ and a fixed positive integer
$m\in\BZ^+,$ we define
$$\nu_{r,s}:\mathscr{P}_{r,m}\lrt \mathscr{P}_{s,m}$$
by
\begin{equation*}\label{Formula 4.4}
  \nu_{r,s}((Y,\alpha)):=\left( \begin{pmatrix}
                       Y & 0 \\
                       0 & I_{s-r}
                     \end{pmatrix}, (\alpha,0)\right),
\end{equation*}
where $Y\in \mathscr{P}_r, \ \alpha\in \BR^{(m,r)}$ and $(\alpha,0)\in \BR^{(m,s)}$.
We let
$$
\mathscr{P}_{\infty,m}:=
\,\lim_{\begin{subarray}{c} \longrightarrow\\ ^r \end{subarray}} \mathscr{P}_{r,m}
$$
be the inductive limit of the directed system $(\mathscr{P}_{r,m},\nu_{r,s}).$
Then $GL_{\infty,m}$ acts on $\mathscr{P}_{\infty,m}$ transitively and
$K_{\infty,m}$ is the stabilizer of this action at the origin. Furthermore
$\G_{\infty,m}$ acts on $\mathscr{P}_{\infty,m}$ properly discontinously.
We put
\begin{equation}\label{Formula 4.8}
  \mathscr{P}_{\infty,m}:=GL_{\infty,m}/K_{\infty,m}
\end{equation}
and
\begin{equation}\label{Formula 4.9}
  \mathfrak{Y}_{\infty,m}:=\G_{\infty,m} \ba GL_{\infty,m}/K_{\infty,m}.
\end{equation}
\vskip 2mm
We may introduce automorphic forms on $\mathscr{P}_{\infty,m}$ for $\G_{\infty,m}$
and study the arithmetic and geometry of $\mathfrak{Y}_{\infty,m}$. In fact,
$\mathfrak{Y}_{\infty,m}$ is important geometrically and number theoretically.

\vskip 3mm
We denote by $\BD({\mathscr P}_{n,m})$ the algebra of all differential operators on
${\mathscr P}_{n,m}$ invariant under the action (4.2) of $GL_{n,m}(\BR)$.
It is known that $\BD({\mathscr P}_{n,m})$ is not commutative (cf.\,\cite{Y3}).

\vskip 3mm
We present several natural problems.
\begin{problem}
Is $\BD({\mathscr P}_{n,m})$ finitely generated ?
\end{problem}

\begin{problem}
Find explicit generators of $\BD({\mathscr P}_{n,m})$.
\end{problem}

\begin{problem}
Find all the relations among the generators of $\BD({\mathscr P}_{n,m})$.
\end{problem}

\begin{remark}
It is true that $\BD({\mathscr P}_{n,m})$ is finitely generated\,({\rm cf.}\,\cite{Y3}).
\end{remark}

\vskip 5.3mm
We present some invariant differential operators on
${\mathscr P}_{n,m}.$ We define the differential operators $D_j,\
 \Omega_{pq}$ and $L_p$ on ${\mathscr P}_{n,m}$ by
\begin{equation}
D_j=\,\textrm{Tr}\left(\left(\,2Y\Yd\right)^j\right),\ \ \ \ 1\leq
j\leq n,\end{equation}

\begin{equation}
\Omega_{pq}^{(k)}=\,\left\{\,{{\partial}\over{\partial V}}\left(
2Y{{\partial}\over{\partial Y}}\right)^kY\,{}^{{}^{{}^{{}^\text{\scriptsize $t$}}}}\!\!\!\left({{\partial}\over
{\partial V}}\right)\,\right\}_{pq},\quad 0\leq k\leq n-1,\ 1\leq p\leq q\leq
m\end{equation} and
\begin{equation}
L_p=\,\textrm{Tr}\left(\left\{Y\,{}^{{}^{{}^{{}^\text{\scriptsize $t$}}}}\!\!\!\left(\Vd\right)\Vd\right\}^p\right),\
\ \ 1\leq p\leq m.\end{equation}
\noindent
Here for a matrix $A$ we denote by $A_{pq}$ the $(p,q)$-entry of $A$.

\begin{remark}
We want to mention the special invariant differential
operator on ${\mathscr P}_{n,m}$. In \cite{Y}, the author studied the
following differential operator $M_{n,m,{\mathcal M}}$ on
${\mathscr P}_{n,m}$ defined by
\begin{equation}
M_{n,m,\CM}=\det\,(Y)\cdot\det\,\left( \Yd+{1\over {8\pi}}\,
{}^{{}^{{}^{{}^\text{\scriptsize $t$}}}}\!\!\!\left(\Vd\right)\CM^{-1}\left(\Vd\right)\right),
\end{equation}
where $\CM$ is a positive definite, symmetric
half-integral matrix of degree $m$. This differential operator
characterizes {\sf singular\ Jacobi\ forms}. For more detail, we
refer to \cite{Y}. According to (4.3) and (4.4), we see easily that
the differential operator $M_{n,m,\CM}$ is invariant under the
action (4.2) of $\Gnm (\BR)$.
\end{remark}

\vskip 3mm
We recall the invariant differential
operators $D_j\, (1\leq j\leq n)$ in (4.10) and
$\Omega_{pq}^{(k)}\ (0\leq k\leq n-1,\ 1\leq p\leq q\leq m)$ in (4.11).

\vskip 0.5cm\noindent
\begin{theorem}
The following relations hold:
\begin{equation}
[D_i,D_j]=0\quad \textrm{for all}\ 1\leq i,j\leq n,
\end{equation}
\begin{equation}
[\Omega_{kl}^{(0)},\Omega_{pq}^{(0)}]=0,\quad 1\leq k\leq l\leq m,\ 1\leq p\leq q\leq m
\end{equation}
and
\begin{equation}
[D_1,\Omega_{pq}^{(0)}]=2\,\Omega_{pq}^{(0)}\quad \textrm{for all}\ \ 1\leq p\leq q \leq m.
\end{equation}
\end{theorem}

\begin{proof}
The relation (4.14) follows from the work of Atle Selberg (cf.\,\cite{M2, S1,T}).
We take a coordinate $(Y,V)$ in ${\mathscr P}_{n,m}$ with $Y=(y_{ij})$ and $V=(v_{kl})$.
We put
$$\Yd=\left( { {1+\delta_{ij}}\over 2 }
{{\partial}\over {\partial y_{ij} }}\right)\quad \textrm{and}\quad \Vd=\left({
{\partial}\over {\partial v_{kl}} } \right),$$ where $1\leq
i,j,l\leq n$ and $1\leq k\leq m.$ Then we get

\begin{eqnarray*}
D_1 &=& 2\,\sum_{1\le i\leq j\leq n} y_{ij}\,
{{\partial}\over {\partial y_{ij} }},\\
\Omega_{pq}^{(0)}&=& \sum_{a=1}^n y_{aa}\,{{\partial^2\qquad}\over {\partial v_{pa}\partial v_{qa}} }\,+\,
\sum_{1\leq a < b\leq n} y_{ab} \left( {{\partial^2\qquad}\over {\partial v_{pa}\partial v_{qb}} }\,+\,
{{\partial^2\qquad}\over {\partial v_{pb}\partial v_{qa}} } \right).
\end{eqnarray*}
By a direct calculation, we obtain the desired relations (4.15) and (4.16).
\end{proof}

\begin{conjecture}
The invariant differential operators $D_j\, (1\leq j\leq n)$and
$\Omega_{pq}^{(k)}\ (0\leq k\leq n-1,\ 1\leq p\leq q\leq m)$ generate
the noncommutative algebra $\BD ({\mathscr P}_{n,m})$.
\end{conjecture}

\begin{conjecture}
The above relations (4.14), (4.15) and (4.16)
generate all the relations among the set
$$\left\{ D_j,\,\Omega_{pq}^{(k)}\,|\
1\leq j\leq n,\ 0\leq k\leq n-1,\ 1\leq p\leq q\leq m \right\}.$$
\end{conjecture}

\begin{problem}
Find a natural way to construct generators of $\BD ({\mathscr P}_{n,m})$.
\end{problem}

\vskip 0.5cm
Using $GL_{n,m}(\BR)$-invariant differential operators on the Minkowski-Euclid space
${\mathscr P}_{n,m},$
we introduce a notion of automorphic forms on ${\mathscr P}_{n,m}$.
\vskip 0.3cm\noindent
We recall the discrete subgroup $\Gamma_{n,m}$ of $GL_{n,m}(\BR)$ defined by
\begin{equation*}
\Gamma_{n,m}:=GL(n,\BZ)\ltimes \BZ^{(m,n)}.
\end{equation*}
Let ${\mathscr Z}_{n,m}$ be the center of $\BD ({\mathscr P}_{n,m})$.

\vskip 0.5cm\noindent
\begin{definition} A smooth function $f:{\mathscr P}_{n,m}\lrt \BC$ is said to be
an {\sf automorphic form} for $\Gamma_{n,m}$ if it satisfies the following conditions
{\rm (GL1)}--{\rm (GL3)}\,:
\vskip 0.2cm\noindent
\ {\rm (GL1)} $f$ is $\Gamma_{n,m}$-invariant\,;
\vskip 0.2cm\noindent
\ {\rm (GL2)} $f$ is an eigenfunction of any differential operator in the center
${\mathscr Z}_{n,m}$ of $\BD ({\mathscr P}_{n,m})$\,;
\vskip 0.2cm\noindent
\ {\rm (GL3)} $f$ has a growth condition, i.e.,
there\ exist\ a constant  $C>0$ and $s\in \BC^{n-1}$ with  \\
\indent \ \ \ \ \ \ \
$s=(s_1,\cdots,s_{n-1})$ such that
$|f(Y,\xi)| \leq C\,|p_{-s}(Y)|$ as the upper left determinants\\
\indent \ \ \ \ \ \ \
$\det Y_j\lrt \infty\ (1\leq j\leq n-1),$ where
\begin{equation*}
  p_{-s}(Y):=\prod_{j=1}^{n-1} (\det Y_j)^{-s_j}
\end{equation*}
\indent \ \ \ \ \ \ \ \
is the Selberg's power function\,(cf.\,\cite{S1, T}).
\end{definition}
\vskip 3mm
We denote by ${\bf A}(\G_{n,m})$ the space of all automorphic forms for $\G_{n,m}$.
A {\sf cusp form} $f\in {\bf A}(\G_{n,m})$ is defined to be an automorphic form
for $\G_{n,m}$ satisfying the following conditions\,:
\begin{equation*}
\int_{(X,\eta)\in \mathfrak{R}}
f \left( Y\left[ \begin{pmatrix}
                   I_j & X \\
                   0 & I_{n-j}
                 \end{pmatrix}\right],\eta\right)dX\,d\eta=0,
                 \quad 1\leq j\leq n-1,
\end{equation*}
where $\mathfrak{R}=(\BR/\BZ)^{(j,n-j)}\times (\BR/\BZ)^{(m,n)},\
(X,\eta)\in \mathfrak{R}$ with $X\in (\BR/\BZ)^{(j,n-j)}$ and
$\eta\in (\BR/\BZ)^{(m,n)}.$ We note that $dX$\,(resp.\ $d\eta$) denotes
the Euclid volume element on $(\BR/\BZ)^{(j,n-j)}$ (resp. $(\BR/\BZ)^{(m,n)}$).
We denote by ${\bf A}_0(\G_{n,m})$ the space of all cusp forms for $\G_{n,m}$.

%Problem 4.5
\begin{problem}
Find the center ${\mathscr Z}_{n,m}$
of $\BD ({\mathscr P}_{n,m})$ explicitly.
\end{problem}

%Problem 4.6
\begin{problem}
Find the center of the universal enveloping algebra of the complexification of the Lie algebra
of the group $GL_{n,m}(\BR)$ explicitly.
\end{problem}

%Problem 4.7
\begin{problem}
Develop the theory of harmonic analysis on $L^2(\G_{n,m}\ba \mathscr{P}_{n,m})$ with respect to
${\mathscr Z}_{n,m}$.
\end{problem}

\vskip 2mm
We may define another notion of automorphic forms as follows.
\begin{definition}
Let $\BD^{\spadesuit}(\mathscr{P}_{n,m})$ be a commutative subalgebra of
$\BD({\mathscr P}_{n,m})$ containing the Laplace operator $\Delta_{n,m;A,B}$
(see Lemma 4.1.) of ${\mathscr P}_{n,m}.$
A smooth function $f:\mathscr{P}_{n,m}\lrt \BC$ is said to be an
{\sf automorphic form} for $\G_{n,m}$ with respect to $\BD^{\spadesuit}(\mathscr{P}_{n,m})$
if it satisfies
the following properties ${\rm (GL1)}^{\spadesuit}$--${\rm (GL3)}^{\spadesuit}$\,:
\vskip 2mm
${\rm (GL1)}^{\spadesuit}$ \ \ $f(\g^*\cdot Y^*)=f(Y^*)\quad {\rm for\ all}\ \g^*\in\G_{n,m}
\ {\rm and}\ Y^*\in \mathscr{P}_{n,m}\,;$
\vskip 2mm
${\rm (GL2)}^{\spadesuit}$ $f$ is an eigenfunction for all differential operators in $\BD^{\spadesuit}(\mathscr{P}_{n,m})\,;$
\vskip 2mm
${\rm (GL3)}^{\spadesuit}$\ There\ exist\ a constant  $C>0$ and $s\in \BC^{n-1}$
with  $s=(s_1,\cdots,s_{n-1})$ such that  \\
\indent \ \ \ \ \ \ \ \ \
$|f(Y,\xi)| \leq C\,|p_{-s}(Y)|$ as the upper left determinants
$\det Y_j\lrt \infty\ (1\leq j\leq n-1),$ \\
\indent \ \ \ \ \ \ \ \ \ where
\begin{equation*}
  p_{-s}(Y):=\prod_{j=1}^{n-1} (\det Y_j)^{-s_j}
\end{equation*}
\indent \ \ \ \ \ \ \ \ \
is the Selberg's power function\,(cf.\,\cite{S1, T}).
\end{definition}

\vskip 3mm
From now on, we fix the commutative subalgebra $\BD^{\spadesuit}(\mathscr{P}_{n,m})$ of
$\BD({\mathscr P}_{n,m})$ containing the Laplace operator $\Delta_{n,m;A,B}$
of ${\mathscr P}_{n,m}.$
We denote by ${\bf A}^{\spadesuit}(\G_{n,m})$ the space of all automorphic forms for $\G_{n,m}$
with respect to $\BD^{\spadesuit}(\mathscr{P}_{n,m})$.
A {\sf cusp form} $f\in {\bf A}^{\spadesuit}(\G_{n,m})$ is defined to be an automorphic form
for $\G_{n,m}$ satisfying the following conditions\,:
\begin{equation*}
\int_{(X,\eta)\in \mathfrak{R}}
f \left( Y\left[ \begin{pmatrix}
                   I_j & X \\
                   0 & I_{n-j}
                 \end{pmatrix}\right],\eta\right)dX\,d\eta=0,
                 \quad 1\leq j\leq n-1,
\end{equation*}
where $\mathfrak{R}=(\BR/\BZ)^{(j,n-j)}\times (\BR/\BZ)^{(m,n)},\
(X,\eta)\in \mathfrak{R}$ with $X\in (\BR/\BZ)^{(j,n-j)}$ and
$\eta\in (\BR/\BZ)^{(m,n)}.$ We note that $dX$\,(resp.\ $d\eta$) denotes
the Euclid volume element on $(\BR/\BZ)^{(j,n-j)}$ (resp. $(\BR/\BZ)^{(m,n)}$).
We denote by ${\bf A}_0^{\spadesuit}(\G_{n,m})$ the space of all cusp forms for $\G_{n,m}$ with respect to $\BD^{\spadesuit}(\mathscr{P}_{n,m})$.

\vskip 3mm
We recall
$$\mathfrak{Y}_{n,m}=\G_{n,m}\ba \mathscr{P}_{n,m}\quad {\rm (see\ Formula\ (4.7))}.$$
Let $L^2(\mathfrak{Y}_{n,m})$ be the Hilbert space of all $L^2$-functions
with $L^2$-norm on $\mathfrak{Y}_{n,m}$.

%Problem 4.8
\begin{problem}
Develop the theory of harmonic analysis on $L^2(\mathfrak{Y}_{n,m})$
with respect to $\BD^{\spadesuit}(\mathscr{P}_{n,m})$. In particular, develop
the spectral theory of the Laplace operator $\Delta_{n,m;A,B}$ of $\mathscr{P}_{n,m}$
on $L^2(\mathfrak{Y}_{n,m})$.
\end{problem}

%Problem 4.9
\begin{problem}
Let $L^2(\G_{n,m}\ba GL_{n,m}(\BR))$ be the Hilbert space of all $L^2$-functions
with $L^2$-norm on $\G_{n,m}\ba GL_{n,m}(\BR)$.
Develop the theory of harmonic analysis on $L^2(\G_{n,m}\ba GL_{n,m}(\BR))$ via
the representation theory of $GL_{n,m}(\BR).$
\end{problem}

%Remark 4.3
\begin{remark}
We refer to \cite{Y2} for the study of unitary representations of $GL_{n,m}(\BR)$.
\end{remark}
\vskip 3mm
A smooth function $f:\mathscr{P}_{n,m}\lrt \BC$ said to be a {\sf weak automorphic form}
for $\G_{n,m}$ if it satisfies the two conditions (GL1),\,(GL3) of Definition 4.1
together the following condition $({\rm GL2})^{\diamondsuit}$\,;
\vskip 2mm
$({\rm GL2})^{\diamondsuit}$ $f$ is an eigenfunction of the Laplace operator $\Delta_{n,m;A,B}$
of $\mathscr{P}_{n,m}$.

\vskip 3mm
Hopefully we may investigate the $L$-functions and Whittaker functions associated to
automorphic forms for $\G_{n,m}$, associated Hecke operator, Fourier expansions of automorphic forms,
Eisenstein series, Poincar{\'e} series and their related topics.

\end{section}

\vskip 9mm

%%%%%%%%%%%%%%%%%%%%%%%%%%%%%%%%%%%%%%%%%%%%%%%%%%%%%%%%%%%%%%%%%%%%%%%%%%%%%%%%%%%%%%%%%%%%%%%%%%%%%%%%%%%%%%%%%%%%%%%%%%%%%%%%%%%%%%%%%%%%%%%%%%%%%%%%
%%%%%%%%%%%%%%%%%%%%%%%%%%%%%%%%%%%%%%%%%%%%%%%%%%%%%%%%%%%%%%%%%%%%%%%%%%%%%%%%%%%%%%%%%%%%%%%%%%%%%%%%%%%%%%%%%%%%%%%%%%%%%%%%%%%%%%%%%%%%%%%%%%%%%%%%
%%%%%%%%%%%%%%%%%%%%%%%%%%%%%%%%%%%%%%%%%%%%%%%%%%%%%%%%%%%%%%%%%%%%%%%%%%%%%%%%%%%%%%%%%%%%%%%%%%%%%%%%%%%%%%%%%%%%%%%%%%%%%%%%%%%%%%%%%%%%%%%%%%%%%%%%
%%%%%%%%%%%%%%%%%%%%%%%%%%%%%%%%%%%%%%%%%%%%%%%%%%%%%%%%%%%%%%%%%%%%%%%%%%%%%%%%%%%%%%%%%%%%%%%%%%%%%%%%%%%%%%%%%%%%%%%%%%%%%%%%%%%%%%%%%%%%%%%%%%%%%%%%
%%%%%%%%
%%%%%%%%
%%%%%%%%                   5. Automorphic Forms for SL(n,\BZ)\ltimes \BZ^{(m,n)}
%%%%%%%%
%%%%%%%%
%%%%%%%%%%%%%%%%%%%%%%%%%%%%%%%%%%%%%%%%%%%%%%%%%%%%%%%%%%%%%%%%%%%%%%%%%%%%%%%%%%%%%%%%%%%%%%%%%%%%%%%%%%%%%%%%%%%%%%%%%%%%%%%%%%%%%%%%%%%%%%%%%%%%%%%%
%%%%%%%%%%%%%%%%%%%%%%%%%%%%%%%%%%%%%%%%%%%%%%%%%%%%%%%%%%%%%%%%%%%%%%%%%%%%%%%%%%%%%%%%%%%%%%%%%%%%%%%%%%%%%%%%%%%%%%%%%%%%%%%%%%%%%%%%%%%%%%%%%%%%%%%%
%%%%%%%%%%%%%%%%%%%%%%%%%%%%%%%%%%%%%%%%%%%%%%%%%%%%%%%%%%%%%%%%%%%%%%%%%%%%%%%%%%%%%%%%%%%%%%%%%%%%%%%%%%%%%%%%%%%%%%%%%%%%%%%%%%%%%%%%%%%%%%%%%%%%%%%%
%%%%%%%%%%%%%%%%%%%%%%%%%%%%%%%%%%%%%%%%%%%%%%%%%%%%%%%%%%%%%%%%%%%%%%%%%%%%%%%%%%%%%%%%%%%%%%%%%%%%%%%%%%%%%%%%%%%%%%%%%%%%%%%%%%%%%%%%%%%%%%%%%%%%%%%%

\begin{section}{{\bf Automorphic Form for $SL(n,\BZ)\ltimes \BZ^{(m,n)}$}}
\setcounter{equation}{0}

\vskip 2mm

For two positive integers $m,\,n\in\BZ^+$, we introduce the new group
$$ SL_{n,m}(\BR):=SL(n,\BR)\ltimes \BR^{(m,n)}$$
which is the semidirect product of $SL(n,\BR)$ and the additive group $\BR^{(m,n)}.$
$SL_{n,m}(\BR)$ is endowed with the following with multiplication law
\begin{equation}\label{Formula (5.1)}
  (A,a)\circ (B,b):=(AB,a\,{}^tB^{-1}+b)
\end{equation}
for all $A,B\in SL(n,\BR)$ and $a,b\in\BR^{(m,n)}$.
We let
$$SL_{n,m}(\BZ):=SL(n,\BZ)\ltimes \BZ^{(m,n)}$$
be a discrete subgroup of $SL_{n,m}(\BR)$.
\vskip 2mm
For any two positive integers $m,\,n\in\BZ^+$, we denote
\begin{equation}
  {\mathfrak P}_{n,m}:={\mathfrak P}_n\times \BR^{(m,n)}.
\end{equation}\label{Formual (5.2)}
Then $SL_{n,m}(\BR)$ acts on ${\mathfrak P}_{n,m}$
naturally and transitively by
\begin{equation}\label{Formula (5.2)}
  (A,a)\cdot (Y,V):=(AY\,{}^tA,(V+a)\,{}^tA)
\end{equation}
for all $(A,a)\in SL_{n,m}(\BR)$ and $(Y,V)\in {\mathfrak P}_{n,m}.$
It is easily seen that the stabilizer of the action (5.3) at $(I_n,0)$ is given by
$$H_{n,m}(\BR):=\{\, (h,0)\in SL_{n,m}(\BR)\,|\ h\in SO(n,\BR)\,\}\cong SO(n,\BR).$$
Thus the non-symmetric homogeneous space $SL_{n,m}(\BR)/SO(n,\BR)$ is diffeomorphic to
the homogeneous space ${\mathfrak P}_{n,m}$ by the following correspondence
$$ (A,a)\cdot SO(n,\BR)\longmapsto (A,a)\cdot (I_n,0)=(A\,{}^t\!A,a\,{}^t\!A)\quad
{\rm for\ all}\ (A,a)\in SL_{n,m}(\BR).$$
We let
$$ \G^n:=SL(n,\BZ)\qquad {\rm and}\qquad \G^{n,m}:=\G^n \ltimes \BZ^{(m,n)}$$
be the discrete arithmetic subgroup of $SL(n,\BR)$ and $SL_{n,m}(\BR)$ respectively. We let
$${\mathfrak X}_n:=\G^n\ba {\mathfrak P}_n$$
be the {\sf special} Minkowski moduli space. If $Y_1,Y_2\in {\mathfrak P}_n$ such that
$Y_2=\gamma Y_1\,{}^t\gamma =Y_1[\,{}^t\gamma]$ for some $\g \in \G^n$, then the real torus
$T(Y_1):=\BR^{(m,n)}/L(Y_1)$ is diffeomorphic to the torus $T(Y_2):=\BR^{(m,n)}/L(Y_2)$,
where $L(Y_1)=\BZ^{(m,n)}Y_1$ and $L(Y_2)=\BZ^{(m,n)}Y_2$ are lattices in the Euclidean space
$\BR^{(m,n)}$. Let
\begin{equation}\label{Formula (5.4)}
  \mathfrak{X}_{n,m}:=\G^{n,m}\ba \mathfrak{P}_{n,m}
\end{equation}
be the universal family of special principally polarized real tori of dimension $mn$.
Roughly we say that $\mathfrak{X}_{n,m}$ is a fibre bundle over
$\mathfrak{X}_n$ with fibres given by special principally polarized real tori.
We refer to \cite{Y3,Y4} for related topics about $\mathfrak{X}_n$ and $\mathfrak{X}_{n,m}.$

\vskip 3mm
For any two positive integers $r,s\in \BZ^+$ with $r<s$ and a fixed positive integer
$m\in\BZ^+,$ we define
$$ \varphi_{r,s}:SL_{r,m}(\BR)\lrt SL_{s,m}(\BR)$$
by
\begin{equation}\label{Formula (5.5)}
  \varphi_{r,s}((A,0)):=\left( \begin{pmatrix}
                       A & 0 \\
                       0 & I_{s-r}
                     \end{pmatrix}, (a,0)\right),
\end{equation}
where $A\in SL(r,\BR),\ a\in \BR^{(m,r)}$ and $(a,0)\in \BR^{(m,s)}.$
\vskip 3mm
For a fixed positive integer $m\in\BZ^+,$ we let
$$
SL_{\infty,m}:=\,\lim_{\begin{subarray}{c} \longrightarrow\\ ^r \end{subarray}}
SL_{r,m}(\BR), \quad
H_{\infty,m}:=\,\lim_{\begin{subarray}{c} \longrightarrow\\ ^r \end{subarray}}
H_{r,m}(\BR)\quad {\rm and}\quad
\G^{\infty,m}:=\,\lim_{\begin{subarray}{c} \longrightarrow\\ ^r \end{subarray}}
\G^{r,m}
$$
be the inductive limits of the directed systems
$(SL_{r,m}(\BR),\varphi_{r,s}),\ (H_{r,m}(\BR),\varphi_{r,s})$ and \\
$(\G^{r,m}(\BR),\varphi_{r,s})$ respectively.
\vskip 3mm
For any two positive integers $r,s\in \BZ^+$ with $r<s$ and a fixed positive integer
$m\in\BZ^+,$ we define
$$\theta_{r,s}:\mathfrak{P}_{r,m}\lrt \mathfrak{P}_{s,m}$$
by
\begin{equation}\label{Formula (5.6)}
  \theta_{r,s}((Y,\alpha)):=\left( \begin{pmatrix}
                       Y & 0 \\
                       0 & I_{s-r}
                     \end{pmatrix}, (\alpha,0)\right),
\end{equation}
where $Y\in \mathfrak{P}_r, \ \alpha\in \BR^{(m,r)}$ and $(\alpha,0)\in \BR^{(m,s)}$.
We let
$$
\mathfrak{P}_{\infty,m}:=
\,\lim_{\begin{subarray}{c} \longrightarrow\\ ^r \end{subarray}} \mathfrak{P}_{r,m}
$$
be the inductive limit of the directed system $(\mathfrak{P}_{r,m},\nu_{r,s}).$
Then $SL_{\infty,m}$ acts on $\mathfrak{P}_{\infty,m}$ transitively and
$H_{\infty,m}$ is the stabilizer of this action at the origin. Furthermore
$\G^{\infty,m}$ acts on $\mathfrak{P}_{\infty,m}$ properly discontinously. We denote
\begin{equation}\label{Formula (5.7)}
  \mathfrak{P}_{\infty,m}:=SL_{\infty,m}/H_{\infty,m}
\end{equation}
and
\begin{equation}\label{Formula (5.8)}
  \mathfrak{X}_{\infty,m}=\G^{\infty,m} \ba SL_{\infty,m}/H_{\infty,m}.
\end{equation}
\vskip 2mm
We may introduce automorphic forms on $\mathfrak{P}_{\infty,m}$ for $\G^{\infty,m}$
and study the arithmetic and geometry of $\mathfrak{X}_{\infty,m}$. In fact,
$\mathfrak{X}_{\infty,m}$ is important geometrically and number theoretically.

\vskip 3mm
We denote by $\BD({\mathfrak P}_{n,m})$ the algebra of all differential operators on
${\mathfrak P}_{n,m}$ invariant under the action (5.3) of $SL_{n,m}(\BR)$.
It is known that $\BD({\mathfrak P}_{n,m})$ is not commutative (cf.\,\cite{Y3}).

\vskip 3mm
Now we present the following natural problems.
%Problem 5.1
\begin{problem}
Is $\BD({\mathfrak P}_{n,m})$ finitely generated ?
\end{problem}

%Problem 5.2
\begin{problem}
Find explicit generators of $\BD({\mathfrak P}_{n,m})$.
\end{problem}

%Problem 5.3
\begin{problem}
Find all the relations among the generators of $\BD({\mathfrak P}_{n,m})$.
\end{problem}

%Remark 5.1
\begin{remark}
It is true that $\BD({\mathfrak P}_{n,m})$ is finitely generated\,({\rm cf.}\,\cite{Y3}).
\end{remark}

%Problem 5.4
\begin{problem}
Find the center ${\mathfrak Z}_{n,m}$
of $\BD ({\mathfrak P}_{n,m})$ explicitly.
\end{problem}

%Problem 5.5
\begin{problem}
Find the center of the universal enveloping algebra of the complexification of the Lie algebra
of the group $SL_{n,m}(\BR)$ explicitly.
\end{problem}

%Problem 5.6
\begin{problem}
Develop the theory of harmonic analysis on $L^2(\G^{n,m}\ba \mathfrak{P}_{n,m})$ with respect to
${\mathfrak Z}_{n,m}$.
\end{problem}

\vskip 3mm
Using $SL_{n,m}(\BR)$-invariant differential operators on the non-symmetric space
${\mathfrak P}_{n,m},$
we introduce a notion of automorphic forms on ${\mathfrak P}_{n,m}$.

%Definition 5.1.
\begin{definition} A smooth function $f:{\mathfrak P}_{n,m}\lrt \BC$ is said to be
an {\sf automorphic form} for $\Gamma^{n,m}$ if it satisfies the following conditions:
\vskip 0.2cm\noindent
\ {\rm (SL1)} $f$ is $\Gamma^{n,m}$-invariant.
\vskip 0.2cm\noindent
\ {\rm (SL2)} $f$ is an eigenfunction of any differential operator in the center
${\mathfrak Z}_{n,m}$ of $\BD ({\mathfrak P}_{n,m})$.
\vskip 0.2cm\noindent
\ {\rm (SL3)} $f$ has a growth condition, i.e.,
there\ exist\ a constant  $C>0$ and $s\in \BC^{n-1}$ with  \\
\indent \ \ \ \ \ \ \
$s=(s_1,\cdots,s_{n-1})$ such that
$|f(Y,\xi)| \leq C\,|p_{-s}(Y)|$ as the upper left determinants\\
\indent \ \ \ \ \ \ \
$\det Y_j\lrt \infty\ (1\leq j\leq n-1),$ where
\begin{equation*}
  p_{-s}(Y):=\prod_{j=1}^{n-1} (\det Y_j)^{-s_j}
\end{equation*}
\indent \ \ \ \ \ \ \ \
is the Selberg's power function\,(cf.\,\cite{S1, T}).
\end{definition}
\vskip 3mm
We denote by ${\bf A}(\G^{n,m})$ the space of all automorphic forms for $\G^{n,m}$.
A {\sf cusp form} $f\in {\bf A}(\G^{n,m})$ is defined to be an automorphic form
for $\G^{n,m}$ satisfying the following conditions\,:
\begin{equation*}
\int_{(X,\eta)\in \mathfrak{R}}
f \left( Y\left[ \begin{pmatrix}
                   I_j & X \\
                   0 & I_{n-j}
                 \end{pmatrix}\right],\eta\right)dX\,d\eta=0,
                 \quad 1\leq j\leq n-1,
\end{equation*}
where $\mathfrak{R}=(\BR/\BZ)^{(j,n-j)}\times (\BR/\BZ)^{(m,n)},\
(X,\eta)\in \mathfrak{R}$ with $X\in (\BR/\BZ)^{(j,n-j)}$ and
$\eta\in (\BR/\BZ)^{(m,n)}.$ We note that $dX$\,(resp.\ $d\eta$) denotes
the Euclid volume element on $(\BR/\BZ)^{(j,n-j)}$ (resp. $(\BR/\BZ)^{(m,n)}$).
We denote by ${\bf A}_0(\G^{n,m})$ the space of all cusp forms for $\G^{n,m}$.

\vskip 0.5cm
We may define another notion of automorphic forms as follows.
\begin{definition}
Let $\BD^{\spadesuit}(\mathfrak{P}_{n,m})$ be a commutative subalgebra of
$\BD({\mathscr P}_{n,m})$ containing the Laplace operator of ${\mathfrak P}_{n,m}.$
A smooth function $f:\mathfrak{P}_{n,m}\lrt \BC$ is said to be an
{\sf automorphic form} for $\G^{n,m}$ with respect to $\BD^{\spadesuit}(\mathfrak{P}_{n,m})$
if the following properties ${\rm (SL1)}^{\spadesuit}$--${\rm (SL3)}^{\spadesuit}$\,:
\vskip 2mm
${\rm (SL1)}^{\spadesuit}$ \ \ $f(\g^*\cdot Y^*)=f(Y^*)\quad {\rm for\ all}\ \g^*\in \G^{n,m}
\ {\rm and}\ Y^*\in \mathfrak{P}_{n,m}\,;$
\vskip 2mm
${\rm (SL2)}^{\spadesuit}$ $f$ is an eigenfunction for all differential operators in $\BD^{\spadesuit}(\mathfrak{P}_{n,m})\,;$
\vskip 2mm
${\rm (SL3)}^{\spadesuit}$\ There\ exist\ a constant  $C>0$ and $s\in \BC^{n-1}$
with  $s=(s_1,\cdots,s_{n-1})$ such that  \\
\indent \ \ \ \ \ \ \ \ \
$|f(Y,\xi)| \leq C\,|p_{-s}(Y)|$ as the upper left determinants
$\det Y_j\lrt \infty\ (1\leq j\leq n-1),$ \\
\indent \ \ \ \ \ \ \ \ \ where
\begin{equation*}
  p_{-s}(Y):=\prod_{j=1}^{n-1} (\det Y_j)^{-s_j}
\end{equation*}
\indent \ \ \ \ \ \ \ \ \
is the Selberg's power function\,(cf.\,\cite{S1, T}).
\end{definition}

\vskip 3mm
From now on, we fix the commutative subalgebra $\BD^{\spadesuit}(\mathfrak{P}_{n,m})$ of
$\BD({\mathfrak P}_{n,m})$ containing the Laplace operator of ${\mathfrak P}_{n,m}.$
We denote by ${\bf A}^{\spadesuit}(\G^{n,m})$ the space of all automorphic forms for $\G^{n,m}$
with respect to $\BD^{\spadesuit}(\mathfrak{P}_{n,m})$.
A {\sf cusp form} $f\in {\bf A}^{\spadesuit}(\G^{n,m})$ is defined to be an automorphic form
for $\G^{n,m}$ satisfying the following conditions\,:
\begin{equation*}
\int_{(X,\eta)\in \mathfrak{R}}
f \left( Y\left[ \begin{pmatrix}
                   I_j & X \\
                   0 & I_{n-j}
                 \end{pmatrix}\right],\eta\right)dX\,d\eta=0,
                 \quad 1\leq j\leq n-1,
\end{equation*}
where $\mathfrak{R}=(\BR/\BZ)^{(j,n-j)}\times (\BR/\BZ)^{(m,n)},\
(X,\eta)\in \mathfrak{R}$ with $X\in (\BR/\BZ)^{(j,n-j)}$ and
$\eta\in (\BR/\BZ)^{(m,n)}.$ We note that $dX$\,(resp.\ $d\eta$) denotes
the Euclid volume element on $(\BR/\BZ)^{(j,n-j)}$ (resp. $(\BR/\BZ)^{(m,n)}$).
We denote by ${\bf A}_0^{\spadesuit}(\G^{n,m})$ the space of all cusp forms for $\G^{n,m}$ with respect to $\BD^{\spadesuit}(\mathfrak{P}_{n,m})$.

\vskip 3mm
We recall
$$\mathfrak{X}_{n,m}=\G^{n,m}\ba \mathfrak{P}_{n,m}\quad {\rm (see\ Formula\ (5.4))}.$$
Let $L^2(\mathfrak{X}_{n,m})$ be the Hilbert space of all $L^2$-functions
with $L^2$-norm on $\mathfrak{X}_{n,m}$.

%Problem 5.7
\begin{problem}
Develop the theory of harmonic analysis on $L^2(\mathfrak{X}_{n,m})$
with respect to $\BD^{\spadesuit}(\mathfrak{P}_{n,m})$. In particular, develop
the spectral theory of the Laplace operator of $\mathfrak{P}_{n,m}$
on $L^2(\mathfrak{X}_{n,m})$.
\end{problem}

%Problem 5.8
\begin{problem}
Let $L^2(\G^{n,m}\ba SL_{n,m}(\BR))$ be the Hilbert space of all $L^2$-functions
with $L^2$-norm on $\G^{n,m}\ba SL_{n,m}(\BR)$.
Develop the theory of harmonic analysis on $L^2(\G^{n,m}\ba SL_{n,m}(\BR))$ via
the representation theory of $SL_{n,m}(\BR).$
\end{problem}

%Remark 5.2
\begin{remark}
We refer to \cite{Y1} for the study of unitary representations of $SL_{2,m}(\BR)$.
\end{remark}
\vskip 3mm
A smooth function $f:\mathfrak{P}_{n,m}\lrt \BC$ said to be a {\sf weak automorphic form}
for $\G^{n,m}$ if it satisfies the two conditions (SL1),\,(SL3) of Definition 5.1
together the following condition $({\rm GL2})^{\diamondsuit}$\,;
\vskip 2mm
$({\rm SL2})^{\diamondsuit}$ $f$ is an eigenfunction of the Laplace operator
of $\mathfrak{P}_{n,m}$.

\vskip 3mm
We put $\mathfrak{H}_{n,m}:=\mathfrak{H}_n\times \BR^{(m,n)}.$ See Definition 3.1
for the definition of $\mathfrak{H}_n$. Let $\BD(\mathfrak{H}_{n,m})$ be the algebra of
all $SL_{n,m}(\BR)$-invariant differential operators on $\mathfrak{H}_{n,m}.$

\vskip 2mm
For any $\nu=(\nu_1,\nu_2,\cdots,\nu_{n-1}),$ we define the function
$\widetilde{I}_\nu:{\mathfrak H}_{n,m}\lrt \BC$ by
\begin{equation}
%\label{}
  \widetilde{I}_\nu (z,v):=\prod_{i=1}^{n-1} \prod_{i=1}^{n-1}y_i^{b_{ij}\nu_j}
  \quad {\rm for\ all}\ z=x\cdot y \in {\mathfrak H}_n \ {\rm and}\ v\in \BR^{(m,n)},
\end{equation}
where
$$  b_{ij}:=\begin{cases}
            \ \ \ ij, & {\rm if }\ \, i+j\leq n \\
            (n-i)(n-j), & {\rm if}\ \, i+j\geq n.
          \end{cases}
$$
We denote by ${\widetilde{\mathfrak Z}}_{n,m}$ the center of ${\mathbb D}({\mathfrak H}_{n,m}).$
Then we see that $\widetilde{I}_\nu(z,v)$ is an eigenfunction of the center
${\widetilde{\mathfrak Z}}_{n,m}$.
Let us write
\begin{equation}\label{Formula (5.10)}
  D\widetilde{I}_\nu(z,v)=\mu_D\cdot \widetilde{I}_\nu(z,v)\quad
  {\rm for\ every}\ D\in {\widetilde{\mathfrak Z}}_{n,m}.
\end{equation}
Since
$$\mu_{D_1D_2}=\mu_{D_1}
\mu_{D_2}\qquad {\rm for\ all}\ D_1,D_2\in
{\widetilde{\mathfrak Z}}_{n,m},$$
the function $\mu_D$ (viewed as a function of $D$) is a character of
${\widetilde{\mathfrak Z}}_{n,m}$.

\vskip 3mm
We recall Definition 3.2 for Maass forms for $\G^n$ defined by Goldfeld\,(see also
\cite[Definition 5.13,\, pp.\,115--116]{Go}).
We generalize these Maass forms for $\G^n$ to Maass forms on $\mathfrak{H}_{n,m}$
for $\G^{n,m}$ in the following way. First of all, we identify $\mathfrak{H}_{n,m}$
with $\mathfrak{P}_{n,m}$.

%Definition 5.3
\begin{definition}
Let $n$ and $m$ be two positive integers with $n\geq 2.$
For any $\nu=(\nu_1,\nu_2,\cdots,\nu_{n-1})$\\
\noindent
$\in \BC^{n-1},$ a smooth function $f:{\mathfrak H}_{n,m}\lrt \BC$ is
said to be a {\sf Maass\ form} on ${\mathfrak H}_{n,m}$ for $\G^{n,m}$ of
type $\nu$ if it satisfies the following conditions {\rm (MF1)}--{\rm (MF3)}\,:
\vskip 2mm\noindent
${\rm (MF1)}\ \ f(\gamma^*\cdot z^*)=f(z^*)\quad {\rm for\ all}\ \gamma^*\in \G^{n,m}
\ {\rm and}\ z^*\in {\mathfrak H}_{n,m};$
\vskip 2mm\noindent
${\rm (MF2)}\ \ Df(z^*)=\mu_D f(z^*)\quad {\rm for\ all}\ D\in \widetilde{\mathfrak{Z}}_{n,m}$
\ {\rm with\ eigenvalue}\ $\mu_D$\ {\rm given\ by}\ {\rm (5.10)};
\vskip 2mm\noindent
${\rm (MF3)}\ \ \int_{(\G^{n,m}\cap U^*)\ba U^*} f(u^*z^*)du^*=0$
\ \ for all subgroups $U^*$ of the form $U^*=U\times \BR^{(m,n)}$ \\
\indent \ \ \ \ \ with
$$U=\left\{
\begin{pmatrix}
  I_{r_1} & * & * & * \\
  0 & I_{r_2} & * & * \\
  0 & 0 & \ddots & * \\
  0 & 0 & 0 & I_{r_b}
\end{pmatrix}
\right\}\subset SL(n,\BR)$$
\indent \ \ \ \ \
with $r_1+r_2+\cdots+r_b=n.$ Here $I_r$ denotes the $r\times r$ identity matrix and
$*$ denotes \\
\indent \ \ \ \ \ arbitrary real matrices.
\end{definition}

\vskip 3mm
Hopefully we may investigate the $L$-functions and Whittaker functions associated to
Maass forms for $\G^{n,m}$, associated Hecke operator, Fourier expansions of automorphic forms,
Eisenstein series, Poincar{\'e} series and their related topics.

\end{section}

%\begin{section}{{\bf Final Remarks}}
%\setcounter{equation}{0}
%\end{section}

\end{document}